\documentclass[12pt]{amsart}
\usepackage{latexsym}
\usepackage{amsmath}
\usepackage{amssymb}
\usepackage{amsfonts} 
\usepackage{eucal} 
\usepackage{cmmib57} 
\usepackage{amsbsy}
\usepackage{amsthm}
\usepackage{amscd}
\usepackage[dvips]{graphics}
\usepackage[all]{xy}

\usepackage{geometry}
\newtheorem{thm}{Theorem}[section]
\newtheorem*{thm*}{Theorem}

\newtheorem{prop}[thm]{Proposition}
\newtheorem{cor}[thm]{Corollary}
\newtheorem*{cor*}{Corollary}

\theoremstyle{definition}
\newtheorem{defn}[thm]{Definition}
\newtheorem{example}[thm]{Example}

\theoremstyle{remark}
\newtheorem{remark}[thm]{Remark}

\newcommand{\mh}{\mbox{MHM}}

\newcommand {\Fa}    {\ensuremath{\mbox{$\mathcal{F}$}}}

\newcommand {\Ha}    {\ensuremath{\mbox{$\mathcal{H}$}}}

\newcommand {\mbc}   {\ensuremath{\mathbb{C}}}

\newcommand {\real}  {\ensuremath{\mathbb{R}}}
\newcommand {\intg}  {\ensuremath{\mathbb{Z}}}

\newcommand {\cplx}  {\ensuremath{\mathbb{C}}}
\newcommand {\rat}   {\ensuremath{\mathbb{Q}}}

\newcommand {\Ext}   {\ensuremath{\operatorname{Ext}}}
\newcommand {\Hom}   {\ensuremath{\operatorname{Hom}}}

\newcommand {\im}    {\operatorname{im}}
\newcommand {\rk}    {\operatorname{rk}}
\newcommand {\interi}{\operatorname{int}}

\newcommand {\emb}   {\ensuremath{\operatorname{emb}}}

\newcommand {\ip}    {\ensuremath{I^{\bar{p}}}}
\newcommand {\iq}    {\ensuremath{I^{\bar{q}}}}
\newcommand {\imi}   {\ensuremath{I^{\bar{m}}}}

\newcommand {\redh}  {\ensuremath{\widetilde{H}}}

\newcommand {\cone}  {\ensuremath{\operatorname{cone}}}

\newcommand {\id}    {\ensuremath{\operatorname{id}}}

\newcommand {\CWkcb}   {\operatorname{\mathbf{CW}}_{k\supset \partial}}
\newcommand {\HoCW}    {\operatorname{\mathbf{HoCW}}}

\begin{document}


\title[Singularities and Intersection Space Homology]{Deformation of Singularities and the Homology of Intersection Spaces}

\author{Markus Banagl}

\address{Mathematisches Institut, Universit\"at Heidelberg,
  Im Neuenheimer Feld 288, 69120 Heidelberg, Germany}

\email{banagl@mathi.uni-heidelberg.de}

\author{Laurentiu Maxim}

\address{Department of Mathematics,
  University of Wisconsin, 480 Lincoln Drive, Madison, WI 53706, USA}

\email{maxim@math.wisc.edu}

\thanks{The first author was in part supported by a research grant of the
 Deutsche Forschungsgemeinschaft. The second author was partially supported by NSF-1005338.}

\date{\today}




\begin{abstract}
While intersection cohomology is stable under small resolutions, both ordinary
and intersection cohomology are unstable under smooth deformation of singularities.
For complex projective algebraic hypersurfaces with an isolated singularity, we show that the first
author's cohomology of intersection spaces is stable under smooth deformations in all degrees
except possibly the middle, and in the middle degree precisely when the monodromy
action on the cohomology of the Milnor fiber is trivial. In many situations, the
isomorphism is shown to be a ring homomorphism induced by a continuous map.
This is used to show that the rational cohomology of intersection spaces can be
endowed with a mixed Hodge structure compatible with Deligne's mixed Hodge
structure on the ordinary cohomology of the singular hypersurface.
 \end{abstract}

\maketitle


\tableofcontents


\section{Introduction}

Given a singular complex algebraic variety $V$, there are essentially two
systematic geometric processes for removing the singularities: one may resolve
them, or one may pass to a smooth deformation of $V$. Ordinary homology is
highly unstable under both processes. This is evident from duality considerations:
the homology of a smooth variety satisfies Poincar\'e duality, whereas the
presence of singularities generally prevents Poincar\'e duality. 
Goresky and MacPherson's middle-perversity intersection cohomology
$IH^\ast (V;\rat)$, as well as Cheeger's $L^2$-cohomology $H^\ast_{(2)} (V)$
do satisfy Poincar\'e duality for singular $V$; thus it makes sense to ask
whether these theories are stable under the above two processes.
The answer is that both are preserved under so-called small resolutions.
Not every variety possesses a small resolution, though it does possess some
resolution. Both $IH^\ast$ and $H^\ast_{(2)}$ are unstable under smooth
deformations. For projective hypersurfaces with isolated singularities,
the present paper answers positively the question: Is there a cohomology
theory for singular varieties, which is stable under smooth deformations?
Note that the smallness condition on resolutions needed for the stability
of intersection cohomology suggests that the class of singularities for which
such a deformation stable cohomology theory exists must also be restricted
by some condition. \\

Let $\bar{p}$ be a perversity in the sense of intersection homology theory.
In \cite{banagl-intersectionspaces}, the first author introduced a
homotopy-theoretic method that assigns to certain types of real
$n$-dimensional stratified topological pseudomanifolds $X$ CW-complexes
\[ \ip X, \]
the \emph{perversity-$\bar{p}$ intersection spaces} of $X$, such that
for complementary perversities $\bar{p}$ and $\bar{q}$, there is a 
Poincar\'e duality isomorphism
\[ \redh^i (\ip X;\rat) \cong \redh_{n-i} (\iq X;\rat) \]
when $X$ is compact and oriented. This method is in particular
applicable to complex algebraic varieties $V$ with isolated singularities,
whose links are simply connected. The latter is a sufficient, but not a 
necessary condition. If $V$ is an algebraic variety, then $\ip V$ will in general
not be algebraic anymore. If $\bar{p} = \bar{m}$ is the lower middle 
perversity, we will briefly write $IX$ for $\imi X$. The groups
\[ HI^\ast_{\bar{p}} (X;\rat) = H^\ast (\ip X;\rat) \]
define a new cohomology theory for stratified spaces, usually not
isomorphic to intersection cohomology $IH^\ast_{\bar{p}}(X;\rat)$.
This is already apparent from the observation that
$HI^\ast_{\bar{p}}(X;\rat)$ is an algebra under cup product, whereas
it is well-known that $IH^\ast_{\bar{p}} (X;\rat)$ cannot generally,
for every $\bar{p}$, be endowed with a $\bar{p}$-internal algebra
structure. Let us put $HI^\ast (X;\rat) = H^\ast (IX;\rat).$ \\

It was pointed out in \cite{banagl-intersectionspaces} that in the context
of conifold transitions, the ranks of $HI^\ast (V;\rat)$ for a singular
conifold $V$ agree with the ranks of $H^\ast (V_s;\rat)$, where $V_s$
is a nearby smooth deformation of $V$; see the table on page 199 and
Proposition 3.6 in \emph{loc. cit.} The main result, Theorem \ref{thm1},
of the present paper is the following Stability Theorem. \\

\noindent \textbf{Theorem.} {\it
Let $V$ be a complex projective hypersurface of complex dimension $n\not= 2$
with one isolated singularity and let $V_s$ be a nearby smooth deformation
of $V$. Then, for all $i<2n$ and $i\not= n,$ we have
\[ \redh^i (V_s;\rat) \cong \redh I^i (V;\rat). \]
Moreover,
\[ H^n (V_s;\rat) \cong HI^n (V;\rat) \]
if, and only if, the monodromy operator acting on the cohomology of the
Milnor fiber of the singularity is trivial.} \\

\noindent The case of a surface $n=2$ is excluded because a general construction
of the intersection space in this case is presently not available.
However, the theory $HI^\ast (V;\real)$ has a de Rham description \cite{banagl-derhamintspace}
by a certain complex of global differential forms on the top stratum of $V$,
which does not require that links be simply connected. Using this description
of $HI^\ast$, the theorem can be extended to the surface case. The
description by differential forms is beyond the scope of this paper and will
not be further discussed here. \\

Let us illustrate the Stability Theorem with a simple example.
Consider the equation
\[ y^2 = x(x-1)(x-s) \]
(or its homogeneous version $v^2 w = u(u-w)(u-s w),$ defining a curve in $\cplx P^2$),
where the complex parameter $s$ is constrained to lie inside the
unit disc, $| s |<1$. For $s \not= 0,$ the equation defines an elliptic curve
$V_s$, homeomorphic to a $2$-torus $T^2$. For $s =0$, a local isomorphism
\[ V = \{ y^2 = x^2 (x-1) \} \longrightarrow \{ \eta^2 = \xi^2 \} \]
near the origin is given by $\xi = x g(x),$ $\eta = y,$ with $g(x) = \sqrt{x-1}$ analytic and
nonzero near $0$. The equation $\eta^2 = \xi^2$ describes a nodal singularity at the origin in $\mbc^2$, whose link
is $\partial I \times S^1$, two circles. Thus $V$ is homeomorphic to a pinched $T^2$ with a meridian
collapsed to a point, or, equivalently, a cylinder $I\times S^1$ with coned-off boundary. The
ordinary homology group $H_1 (V;\intg)$ has rank one, generated by the longitudinal circle.
The intersection homology group $IH_1 (V;\intg)$ agrees with the intersection homology of the
normalization $S^2$ of $V$:
\[ IH_1 (V;\intg) = IH_1 (S^2;\intg) = H_1 (S^2;\intg)=0. \]
Thus, as $H_1 (V_s;\intg)=H_1 (T^2;\intg) = \intg \oplus \intg$,
neither ordinary homology nor intersection homology remains invariant under the
smoothing deformation $V \leadsto V_s$. The middle perversity intersection space $IV$
of $V$ is a cylinder $I \times S^1$ together with an interval, whose one endpoint is
attached to a point in $\{ 0 \} \times S^1$ and whose other endpoint is attached to a
point in $\{ 1 \} \times S^1$. Thus $IV$ is homotopy equivalent to the figure eight and
\[ H_1 (IV;\intg) = \intg \oplus \intg, \]
which does agree with $H_1 (V_s;\intg)$. Several other examples are worked out
throughout the paper, including a reducible curve, a Kummer surface and quintic
threefolds with nodal singularities. \\

We can be more precise about the isomorphisms of the Stability Theorem.
Given $V$, there is a canonical map $IV\to V$, and given a nearby smooth
deformation $V_s$ of $V$, one has the specialization map
$V_s \to V$. In Proposition \ref{map}, we construct a map $IV\to V_s$
such that $IV\to V_s \to V$ is a factorization of $IV\to V$.
The map $IV \to V_s$ induces the isomorphisms of the Stability Theorem.
It follows in particular that one has an algebra isomorphism
$\redh I^\ast (V;\rat)\cong \redh^\ast (V_s;\rat)$ (in degrees less than
$2n$, and for trivial monodromy). We use this geometrically induced
isomorphism to show that under the hypotheses of the proposition,
$HI^\ast (V;\rat)$ can be equipped with a mixed Hodge structure, so that
$IV\to V$ induces a homomorphism of mixed Hodge structures on
cohomology (Corollary \ref{cor.hodgestruct}). \\

The relationship between $IH^\ast$ and $HI^\ast$ is very well illuminated
by mirror symmetry, which tends to exchange resolutions and deformations.
It is for instance conjectured in \cite{morrison} that the mirror of a conifold
transition, which consists of a deformation $s \to 0$ (degeneration smooth to
singular) followed by a small resolution, is again a conifold transition, but
performed in the reverse direction. This observation strongly suggests that
since there is a theory $IH^\ast$ stable under small resolutions, there ought to
be a mirror theory $HI^\ast$ stable under certain ``small" deformations. This is
confirmed by the present paper and by the results of Section 3.8 in
\cite{banagl-intersectionspaces}, where it is shown that if $V^\circ$ is the mirror
of a conifold $V$, both sitting in  mirror symmetric conifold transitions, then
\[ \begin{array}{lcl} 
\rk IH_3 (V) & = & \rk HI_2 (V^\circ) + \rk HI_4 (V^\circ)+2, \\
\rk IH_3 (V^\circ) & = & \rk HI_2 (V) + \rk HI_4 (V)+2, \\
\rk HI_3 (V) & = & \rk IH_2 (V^\circ) + \rk IH_4 (V^\circ)+2,
 \text{ and} \\ 
\rk HI_3 (V^\circ) & = & \rk IH_2 (V) + \rk IH_4 (V)+2.
\end{array} \]
In the same spirit, the well-known fact that the intersection homology of a complex variety $V$ is a vector subspace of the ordinary homology of any resolution of $V$ is ``mirrored" by our result proved in Theorem \ref{isomap} below, stating that the intersection space homology $HI_* (V)$ is a subspace of the homology $H_* (V_s)$ of any smoothing $V_s$ of $V$.

Since mirror symmetry is a phenomenon that arose originally in string theory,
it is not surprising that the theories $IH^\ast,$ $HI^\ast$ have a specific
relevance for type IIA, IIB string theories, respectively.
While $IH^\ast$ yields the correct count of massless $2$-branes on a conifold
in type IIA theory, the theory $HI^\ast$ yields the correct count of massless
$3$-branes on a conifold in type IIB theory. These are Propositions 3.6, 3.8 and
Theorem 3.9 in \cite{banagl-intersectionspaces}. \\

The Euler characteristics $\chi$ of $IH^\ast$ and $HI^\ast$ are compared in 
Corollary 4.6; the result is seen to be consistent with the formula
\[ \chi(H_*(V))-\chi(IH_*(V))=\sum_{x \in {\rm Sing}(V)} 
   \left(1-\chi(IH_*(\overset{\circ}{\cone} L_x)) \right), \]
where $\overset{\circ}{\cone} L_x$ is the open cone on the link $L_x$ of the singularity $x$, obtained in \cite{CMS08}.
The behavior of classical intersection homology under deformation of singularities is
discussed from a sheaf-theoretic viewpoint in Section \ref{deform}.
Proposition \ref{thm2} observes that the perverse self-dual sheaf
$\psi_{\pi}(\rat_X)[n],$ where $\psi_\pi$ is the nearby cycle functor of a 
smooth deforming family $\pi: X\to S$ with singular fiber $V=\pi^{-1}(0)$,
is isomorphic in the derived category of $V$ to the intersection chain sheaf
$IC_V$ if, and only if, $V$ is nonsingular. The hypercohomology of
$\psi_{\pi}(\rat_X)[n]$ computes the cohomology of the general fiber $V_s$
and the hypercohomology of $IC_V$ computes $IH^\ast (V)$. \\

Finally, the phenomena described in this paper seem to have a wider scope
than hypersurfaces. The conifolds and Calabi-Yau threefolds investigated in
\cite{banagl-intersectionspaces} were not assumed to be hypersurfaces,
nevertheless $HI^\ast$ was seen to be stable under the deformations
arising in conifold transitions. \\

\textbf{Notation.} Rational homology will be denoted by $H_\ast (X), IH_\ast (X), HI_\ast (X)$, whereas
integral homology will be written as $H_\ast (X;\intg), IH_\ast (X;\intg), HI_\ast (X;\intg)$.
The linear dual of a rational vector space $W$ will be written as
$W^\ast = \Hom (W,\rat)$. For a topological space $X$, $\redh_\ast (X)$ and
$\redh^\ast (X)$ denote reduced (rational) homology and cohomology, respectively.


\section{Background on Intersection Spaces}\label{background}

In \cite{banagl-intersectionspaces}, the first author introduced a method that
associates to certain classes of stratified
pseudomanifolds $X$ CW-complexes
\[ \ip X, \]
the \emph{intersection spaces of $X$}, where $\bar{p}$ is a perversity in the
sense of Goresky and MacPherson's intersection homology,
such that the \emph{ordinary} (reduced, rational)
homology $\redh_\ast (\ip X)$ satisfies generalized Poincar\'e duality
when $X$ is closed and oriented.
The resulting homology
theory $X \leadsto HI^{\bar{p}}_\ast (X) = H_\ast (\ip X)$ is neither isomorphic to intersection
homology, which we will write as $IH^{\bar{p}}_\ast (X)$, nor (for real coefficients) linearly dual to $L^2$-cohomology
for Cheeger's conical metrics. The 
Goresky-MacPherson intersection chain complexes $IC^{\bar{p}}_\ast (X)$
are generally not algebras, unless $\bar{p}$ is the zero-perversity, in
which case $IC^{\bar{p}}_\ast (X)$ is essentially the ordinary cochain complex
of $X$. (The Goresky-MacPherson intersection product raises perversities
in general.) Similarly, the differential complex $\Omega^\ast_{(2)}(X)$
of $L^2$-forms on the top stratum is not an algebra under
wedge product of forms. Using the intersection
space framework, the ordinary cochain complex $C^\ast (\ip X)$ of 
$\ip X$ \emph{is} a DGA, simply by employing the ordinary cup product. 
The theory $HI^\ast$ also addresses questions in type II
string theory related to the existence of massless D-branes arising in the
course of a Calabi-Yau conifold transition. These questions are answered by
$IH^\ast$ for IIA theory, and by $HI^\ast$ for IIB theory; see Chapter 3 of
\cite{banagl-intersectionspaces}. 
Furthermore, given a spectrum $E$ in the sense of stable homotopy theory,
one may form $EI^\ast_{\bar{p}} (X) = E^\ast (\ip X)$. This, then, yields an
approach to defining intersection versions of generalized cohomology theories
such as K-theory.

\begin{defn}
The category $\CWkcb$ of
\emph{$k$-boundary-split CW-complexes} consists of the following
objects and morphisms: Objects are pairs $(K,Y)$, where
$K$ is a simply connected CW-complex and $Y \subset C_k (K;\intg)$ is
a subgroup of the $k$-th cellular chain group of $K$ that arises
as the image $Y = s(\im \partial)$ of some splitting
$s: \im \partial \rightarrow C_k (K;\intg)$ of the boundary map
$\partial: C_k (K;\intg) \rightarrow \im \partial (\subset C_{k-1} (K;\intg))$.
(Given $K$, such a splitting always exists, 
since $\im \partial$ is free abelian.) A morphism
$(K,Y_K) \rightarrow (L,Y_L)$ is a cellular map
$f: K \rightarrow L$ such that $f_\ast (Y_K)\subset Y_L$.
\end{defn}
Let $\HoCW_{k-1}$ denote the category whose objects are
CW-complexes and whose morphisms are rel $(k-1)$-skeleton homotopy
classes of cellular maps. Let
\[ t_{<\infty}: \CWkcb \longrightarrow
   \HoCW_{k-1} \]
be the natural projection functor, that is,
$t_{<\infty} (K,Y_K) = K$ for an object $(K, Y_K)$ in 
$\CWkcb$, and $t_{<\infty} (f) = [f]$ for a morphism
$f: (K,Y_K) \rightarrow (L,Y_L)$ in $\CWkcb$.
The following theorem is proved in \cite{banagl-intersectionspaces}.
\begin{thm} \label{thm.generaltruncation}
Let $k\geq 3$ be an integer.
There is a covariant assignment 
$t_{<k}: \CWkcb 
 \longrightarrow \HoCW_{k-1}$ of objects and morphisms
together with a natural transformation
$\emb_k: t_{<k} \rightarrow t_{<\infty}$ such that for an object
$(K, Y)$ of $\CWkcb,$ one has
$H_r (t_{<k} (K,Y);\intg)=0$ for $r\geq k,$ and
\[ \emb_k (K,Y)_\ast: H_r (t_{<k} (K,Y);\intg) \stackrel{\cong}{\longrightarrow}
  H_r (K;\intg) \]
is an isomorphism for $r<k.$ 
\end{thm}
This means in particular that given a morphism $f$, one has squares
\[ \xymatrix{
t_{<k}(K,Y_K)  \ \ \ar[r]^{\emb_k (K,Y_K)} \ar[d]_{t_{<k}(f)}  & 
 \  \ t_{<\infty} (K,Y_K) \ar[d]^{t_{<\infty}(f)} \\
t_{<k}(L,Y_L) \ \ \ar[r]^{\emb_k (L,Y_L)} & \  \
t_{<\infty} (L,Y_L)
} \]
that commute in $\HoCW_{k-1}$.
If $k\leq 2$ (and the CW-complexes are simply connected), then it is of course
a trivial matter to construct such truncations. \\

Let $\bar{p}$ be a perversity.
Let $X$ be an $n$-dimensional compact oriented pseudomanifold with
isolated singularities $x_1, \ldots, x_w,$ $w\geq 1.$ We assume the complement
of the singularities to be a smooth manifold. Furthermore, to be able to apply
the general spatial truncation Theorem \ref{thm.generaltruncation}, we require
the links $L_i = \operatorname{Link}(x_i)$ to be simply connected. This assumption
is not always necessary, as in many non-simply connected situations, ad hoc truncation constructions can
be used.
The $L_i$ are closed smooth
manifolds and a small neighborhood of $x_i$ is homeomorphic to the open
cone on $L_i$.
Every link $L_i$, $i=1,\ldots,w,$ can be given the structure of a CW-complex.
If $k=n-1-\bar{p}(n)\geq 3,$ we can and do fix 
completions $(L_i,Y_i)$ of $L_i$ so that every $(L_i,Y_i)$ is an object in
$\CWkcb$. If $k\leq 2,$ no groups $Y_i$ have to be chosen.
Applying the truncation $t_{<k}: \CWkcb \rightarrow \HoCW_{k-1}$, we obtain a CW-complex
$t_{<k} (L_i,Y_i) \in Ob \HoCW_{k-1}$.
The natural transformation $\emb_k: t_{<k} \rightarrow t_{<\infty}$
of Theorem \ref{thm.generaltruncation} gives homotopy classes $\emb_k (L_i, Y_i)$ 
represented by maps
\[ f_i: t_{<k}(L_i, Y_i) \longrightarrow L_i \]
such that for $r<k,$ 
\[ f_{i\ast}: H_r (t_{<k}(L_i, Y_i)) \cong H_r (L_i), \]
while $H_r (t_{<k}(L_i, Y_i))=0$ for $r\geq k$.
Let $M$ be the compact manifold with boundary obtained by removing from $X$
open cone neighborhoods of the singularities $x_1, \ldots, x_w$.
The boundary is the disjoint union of the links,
\[ \partial M = \bigsqcup_{i=1}^w L_i. \]
Let 
\[ L_{<k} = \bigsqcup_{i=1}^w t_{<k}(L_i, Y_i) \]
and define a map
\[ g: L_{<k} \longrightarrow M \]
by composing
\[ L_{<k} \stackrel{f}{\longrightarrow} \partial M \longrightarrow M, \]
where $f = \bigsqcup_i f_i$. The intersection space is the
homotopy cofiber of $g$:

\begin{defn} \label{def.intspisol}
The \emph{perversity $\bar{p}$ intersection space}
$\ip X$ of $X$ is defined to be
\[ \ip X = \operatorname{cone}(g) = M \cup_g \cone (L_{<k}). \]
\end{defn}

Thus, to form the intersection space, we attach the cone on a suitable spatial homology
truncation of the link to the exterior of the singularity along the
boundary of the exterior. The two extreme cases of this construction
arise when $k=1$ and when $k$ is larger than the dimension of the link.
In the former case, assuming $w=1$, $t_{<1} (L)$ is a point 
and thus $\ip X$ is homotopy equivalent to the 
nonsingular top stratum of $X$. In the latter case
no actual truncation has to be performed, $t_{<k}(L)=L$,
$\emb_k (L)$ is the identity map and thus $\ip X = X$ (again assuming $w=1$).
If the singularities are not
isolated, one attempts to do fiberwise spatial homology truncation applied
to the link bundle. Such fiberwise truncation may be obstructed, however.
If $\bar{p} = \bar{m}$ is the lower middle perversity, then we shall
briefly write $IX$ for $\imi X$. We shall put $HI^{\bar{p}}_\ast (X) =
H_\ast (\ip X)$ and $HI_\ast (X) = H_\ast (IX)$; similarly for cohomology.
When $X$ has only one singular point, there are canonical homotopy classes of maps
\[ M \longrightarrow IV \longrightarrow V \]
described in Section 2.6.2 of \cite{banagl-intersectionspaces}. The first class can
be represented by the inclusion $M\hookrightarrow IV$. A particular representative
$\gamma: IV\to V$ of the second class is described in the proof of
Proposition \ref{map}. If $V$ has several isolated singular points, the target of the
second map has to be slightly modified by identifying all the singular points.
If $V$ is connected, then this only changes the first homology. The intersection
homology does not change at all.
Two perversities $\bar{p}$ and $\bar{q}$ are called \emph{complementary}
if $\bar{p}(s)+\bar{q}(s)=s-2$ for all $s=2,3,\ldots$. The following result
is established in \emph{loc. cit}.
\begin{thm} (Generalized Poincar\'e Duality.)
 Let $\bar{p}$ and $\bar{q}$ be complementary
perversities. There is a nondegenerate intersection form
\[ \redh I^{\bar{p}}_i (X) \otimes
  \redh I^{\bar{q}}_{n-i}(X) \longrightarrow \rat \]
which is compatible with the intersection form on the exterior of the singularities.
\end{thm}
The following formulae
for $\redh I^{\bar{p}}_\ast (X)$ are available (recall $k=n-1-\bar{p}(n)$):
\[ \redh I^{\bar{p}}_i (X)= \begin{cases}
  H_i (M),& i>k \\ H_i (M,\partial M),& i<k. \end{cases} \]
In the cutoff-degree $k$, we have a T-diagram with exact row and exact column:
\[ \small
\xymatrix@C=10pt{
& & 0 \ar[d] & & \\
0 \ar[r] & \ker (H_k (M) \to H_k (M,L)) \ar[r] &
H_k (M) \ar[r] \ar[d] & IH_k (X) \ar[r] & 0 \\
& & HI_k (X) \ar[d] & & \\
& & \im (H_k (M,L) \to H_{k-1} (L)) \ar[d] & & \\
& & 0 & &
} \]
\normalsize
The cohomological version of this diagram is
\[ \small
\xymatrix@C=10pt{
& & 0 \ar[d] & & \\
0 \ar[r] & \ker (H^k (M,L) \to H^k (M)) \ar[r] &
H^k (M,L) \ar[r] \ar[d] & IH^k (X) \ar[r] & 0 \\
& & HI^k (X) \ar[d] & & \\
& & \im (H^k (M) \to H^k (L)) \ar[d] & & \\
& & 0 & &
} \]
\normalsize
When $X$ is a complex variety of complex dimension $n$ and $\bar{p}=\bar{m}$, then
$k=n$. If, moreover, $n$ is even, it was shown in \cite{banagl-intersectionspaces}[Sect.2.5] that the Witt elements (over the rationals) corresponding to the intersection form on $IX$ and, respectively,  the Goresky-MacPherson intersection pairing on the middle intersection homology group,  coincide. In particular, the signature $\sigma(IX)$ of the intersection space equals the Goresky-MacPherson intersection homology signature of $X$. For results comparing the Euler characteristics of the two theories, see \cite{banagl-intersectionspaces}[Cor.2.14] and Proposition \ref{prop1} below.


\section{Background on Hypersurface Singularities}
\label{sec.backgroundsing}

Let $f$ be a homogeneous polynomial in $n+2$ variables with complex coefficients such that the
complex projective hypersurface
\[ V = V(f) = \{ x\in \mathbb{P}^{n+1} ~|~ f(x)=0 \} \]
has one isolated singularity $x_0$. 
Locally, identifying $x_0$ with the origin of $\cplx^{n+1}$, the singularity is described
by a reduced analytic function germ
\[ g: (\cplx^{n+1},0)\longrightarrow (\cplx,0). \]
Let $B_\epsilon \subset \cplx^{n+1}$ be a closed ball of radius $\epsilon >0$ centered at the origin and let
$S_\epsilon$ be its boundary, a sphere of dimension $2n+1$. 
Choose $\epsilon$ small enough so that \\

\noindent (1) the intersection $V\cap B_\epsilon$ is homeomorphic to the cone over the link
$L_0 = V\cap S_\epsilon = \{ g=0 \} \cap S_\epsilon$ of the singularity $x_0$, and \\

\noindent (2) the Milnor map of $g$ at radius $\epsilon,$
\[ \frac{g}{|g|}: S_\epsilon - L_0 \longrightarrow S^1, \]
is a (locally trivial) fibration. \\

\noindent The link $L_0$ is an $(n-2)$-connected $(2n-1)$-dimensional submanifold
of $S_\epsilon$.
The fibers of the Milnor map are open smooth manifolds of real dimension
$2n$. Let $F_0$ be the closure in $S_{\epsilon}$ of the fiber of $g/|g|$ over $1\in S^1$.
Then $F_0$, the \emph{closed Milnor fiber of the singularity} is a compact manifold with 
boundary $\partial F_0 = L_0,$ the link of $x_0$.
Via the fibers of the Milnor map as pages, $S_\epsilon$ receives an open book
decomposition with binding $L_0$. \\

Let $\pi: X\to S$ be a smooth deformation of $V$, where $S$ is a small disc of radius,
say, $r>0$ centered at the origin of $\cplx$. The map $\pi$ is assumed to be proper.
The singular variety $V$ is the special fiber $V=\pi^{-1}(0)$ and the general fibers
$V_s = \pi^{-1}(s),$ $s\in S,$ $s\not= 0$, are smooth projective $n$-dimensional
hypersurfaces. The space $X$ is a complex manifold of dimension $n+1$. Given $V$ as above, we shall show
below that such a smooth deformation $\pi$ can always be constructed.
Let $B_\epsilon (x_0)$ be a small closed ball in $X$ about the singular point $x_0$
such that
\begin{enumerate}
\item $B_\epsilon (x_0) \cap V$ can be identified with the cone on $L_0$,
\item $F= B_\epsilon (x_0) \cap V_s$ can be identified with $F_0$.
\end{enumerate}
(Note that this ball $B_\epsilon (x_0)$ is different from the ball $B_\epsilon$ used
above: the former is a ball in $X$, while the latter is a ball in $\mathbb{P}^{n+1}$.)
Let $B = \interi B_\epsilon (x_0)$ and
let $M_0$ be the compact manifold $M_0 = V - B$ with boundary 
$\partial M_0 = L_0$. For $0<\delta <r,$ set 
$S_\delta = \{ z\in S ~|~ |z| <\delta \}$,
$S^\ast_\delta = \{ z\in S ~|~ 0<|z| <\delta \}$
and $N_\delta = \pi^{-1}(S_\delta) - B$.
Choose $\delta \ll \epsilon$ so small that
\begin{enumerate}
\item $\pi|: N_\delta \to S_\delta$ is a proper smooth submersion and
\item $\pi^{-1}(S_\delta)\subset N:= N_\delta \cup B$.
\end{enumerate}
For $s\in S^\ast_\delta,$ we shall construct the \emph{specialization map}
\[ r_s: V_s \longrightarrow V. \]
By the Ehresmann fibration theorem, 
$\pi|: N_\delta \to S_\delta$ is a locally trivial fiber bundle projection.
Since $S_\delta$ is contractible, this is a trivial bundle, that is, there exists
a diffeomorphism $\phi: N_\delta \to S_\delta \times M_0$ (recall that
$M_0$ is the fiber of $\pi|$ over $0$) such that
\[ \xymatrix{
N_\delta \ar[rr]^{\phi}_{\cong} \ar[rd]_{\pi|} & & S_\delta \times
  M_0 \ar[ld]^{\pi_1} \\
& S_\delta & 
} \]
commutes. The second factor projection $\pi_2: S_\delta \times M_0 \to M_0$
is a deformation retraction. Hence $\rho_\delta = \pi_2 \phi: N_\delta \to M_0$
is a homotopy equivalence.
Let $M$ be the compact manifold $M = V_s - B$ with boundary
$L := \partial M$, $s\in S^\ast_\delta$.
We observe next that the composition
\[ M \hookrightarrow N_\delta \stackrel{\rho_\delta}{\longrightarrow} M_0 \]
is a diffeomorphism. Indeed,
\[ M= \pi|^{-1} (s) = \phi^{-1} \pi^{-1}_1 (s) = \phi^{-1}(\{ s \} \times M_0) \]
is mapped by $\phi$ diffeomorphically onto $\{ s \} \times M_0,$ which is then
mapped by $\pi_2$ diffeomorphically onto $M_0$. This fixes a diffeomorphism
\[ \psi: (M,L) \stackrel{\cong}{\longrightarrow} (M_0, L_0). \]
Thus $L$ is merely a displaced copy of the link $L_0$ of
$x_0$ and $M$ is a displaced copy of the exterior $M_0$ of the singularity $x_0$.
The restricted homeomorphism $\psi|: L\to L_0$ can be levelwise extended to a
homeomorphism $\cone (\psi|): \cone L \to \cone L_0$; the cone point is mapped to
$x_0$. We obtain a commutative diagram
\[ \xymatrix{
M \ar[d]_{\psi}^{\cong} & L \ar[l] \ar[r] \ar[d]_{\psi|}^{\cong} & \cone (L) 
  \ar[d]_{\cone (\psi|)}^\cong  \\
M_0 & L_0 \ar[l] \ar[r] & \cone (L_0),
} \]
where the horizontal arrows are all inclusions as boundaries.
Let us denote by $W:=M\cup_L \cone L$ the pushout of the top row. Since $V$ is topologically the pushout of the bottom row,
$\psi$ induces a homeomorphism
\[ W \stackrel{\cong}{\longrightarrow} V. \]
We think of $W$ as a displaced copy of $V$, and shall work primarily with this topological model of 
$V$. We proceed with the construction of the specialization map.
Using a collar, we may write $M_0$ as $M_0 = \overline{M}_0 \cup [-1,0]\times L_0$ with 
$\overline{M}_0$ a compact codimension $0$ submanifold of $M_0$, which is diffeomorphic to $M_0$.
The boundary of $\overline{M}_0$ corresponds to $\{ -1 \} \times L_0.$ Our model
for the cone on a space $A$ is $\cone (A) = [0,1]\times A / \{ 1 \} \times A.$
The specialization map $r_s: V_s \to V$ is a composition
\[ V_s \hookrightarrow N \stackrel{\rho}{\longrightarrow} V, \]
where $\rho$ is a homotopy equivalence to be constructed next.
On $\rho^{-1}_\delta (\overline{M}_0) \subset N,$ $\rho$ is given by $\rho_\delta$.
The ball $B$ is mapped to the singularity $x_0$. The remaining piece
$C=\rho^{-1}_\delta ([-1,0] \times L_0) \subset N_\delta \subset N$ is mapped to
$[-1,0]\times L_0 \cup \cone (L_0)$ $= [-1,1] \times L_0 / \{ 1 \} \times L_0$
by stretching $\rho_\delta$ from $[-1,0]$ to $[-1,1]$. In more detail: if
$\rho_\delta|: C \to [-1,0]\times L_0$ is given
by $\rho_\delta (x) = (f(x), g(x))$ for smooth maps $f:C\to [-1,0],$ and
$g:C \to L_0,$ then $\rho$ is given on $C$ by
\[ \rho (x) = (h(f(x)), g(x)), \]
where $h: [-1,0]\to [-1,1]$ is a smooth function such that $h(t)=t$ for $t$ close to
$-1$ and $f(t)=1$ for $t$ close to $0$.
This yields a continuous map $\rho: N\to V$ and finishes the construction of the 
specialization map.\\

We shall now show how a smooth deformation as above can be constructed, given
a homogeneous polynomial $f:\mbc^{n+2} \to \mbc$, of degree $d$, defining a complex projective hypersurface 
\[ V = V(f) = \{ x\in \mathbb{P}^{n+1} ~|~ f(x)=0 \} \] with only isolated singularities $p_1, \cdots, p_r$. (We allow here more than just one isolated singularity, as this more general setup will be needed later on.) For each $i \in \{1,\cdots,r\}$, let \[ g_i: (\cplx^{n+1},0)\longrightarrow (\cplx,0). \] be a local equation for $V(f)$ near $p_i$. Let $l$ be a linear form on $\mbc^{n+2}$ such that the corresponding hyperplane in $\mathbb{P}^{n+1}$ does not pass through any of the points $p_1,\cdots, p_r$. By Sard's theorem, there exists $r>0$ so that for any $s \in \mbc$ with $0< \vert s \vert < r$, the hypersurface 
\[ V_s := V(f+s \cdot l^d) \subset \mathbb{P}^{n+1}  \] is non-singular. Define 
\[ X:= \bigcup_{s \in  S} \left( \{ s \} \times V_s \right) \subset S \times \mathbb{P}^{n+1} \]  
with $\pi:X \to S$ the corresponding projection map. Then $X$ is a complex manifold of dimension $n+1$, and  for each $i \in \{1,\cdots, r\}$ the germ of the proper holomorphic map $\pi$ at $p_i$ is equivalent to $g_i$. Note that $\pi$ is smooth over the punctured disc  $S^{\ast} :=\{s \in \mbc  ~|~  0< \vert s \vert < r \}$, as $s=0$ is the only critical value of $\pi$.  Moreover, the fiber $\pi^{-1}(0)$ is the hypersurface $V$, and for any $s \in  S^{\ast},$ the corresponding fiber $\pi^{-1}(s)=V_s$ is a smooth $n$-dimensional complex projective hypersurface of degree $d$. Therefore, each of these $V_s$ ($s \in  S^{\ast}$) can be regarded as a smooth deformation of the given hypersurface $V=V(f)$.\\

Let us collect some facts and tools concerning the Milnor fiber $F \cong F_0$ of an isolated hypersurface singularity germ (e.g., see \cite{dimcasingtophyper,Mi}).
It is homotopy equivalent to a bouquet of $n$-spheres. The number $\mu$ of spheres in
this bouquet is called the {\it Milnor number} and can be computed as
\[ \mu = \dim_\cplx \frac{\mathcal{O}_{n+1}}{J_g}, \]
with $\mathcal{O}_{n+1} = \cplx \{ x_0, \ldots, x_n \}$ the $\cplx$-algebra of all convergent
power series in $x_0, \ldots, x_n$, and $J_g = (\partial g/\partial x_0, \ldots,
\partial g/\partial x_n)$ the Jacobian ideal of the singularity. The inclusion
$L =\partial F \hookrightarrow F$ is an $(n-1)$-equivalence; in particular
\[ H_i (L) \longrightarrow H_i (F) \]
is an isomorphism for $i<n-1$.
The specialization map $r_s:V_s \to V$ induces on homology the
\emph{specialization homomorphism}
\[ H_\ast (V_s)\longrightarrow H_\ast (V). \]
This fits into an exact sequence
\begin{equation} \label{equ.specialsequ}
0 \to H_{n+1} (V_s)\to H_{n+1}(V)\to H_n (F)\to H_n (V_s) \to H_n (V) \to 0,
\end{equation}
which describes the effect of the deformation on homology in degrees $n, n+1$.
Of course,
\[ H_n (F) \cong H_n (\bigvee^\mu S^n) \cong \rat^\mu. \]
In degrees $i\not= n,n+1,$ the specialization homomorphism is an isomorphism
\[ H_i (V_s) \cong H_i (V). \]
Associated with the Milnor fibration $\interi F_0 \hookrightarrow S_\epsilon - L_0 \to S^1$ is a 
monodromy homeomorphism $h_0: \interi F_0 \to \interi F_0$. Using the identity
$L_0 \to L_0,$ $h_0$ extends to a homeomorphism
$h: F_0 = \interi F_0 \cup L_0 \to F_0$ because $L_0$ is the binding of the open
book decomposition. This homeomorphism induces the \emph{monodromy operator}
\[ T = h_\ast: H_\ast (F_0) \stackrel{\cong}{\longrightarrow} H_\ast (F_0). \]
The difference between the monodromy operator and the identity fits into the
\emph{Wang sequence} of the fibration,
\begin{equation} \label{equ.wang}
0 \to H_{n+1} (S_\epsilon - L_0) \to H_n (F_0) 
 \stackrel{T-1}{\longrightarrow} H_n (F_0) \to
 H_n (S_\epsilon - L_0)\to 0, \ \ \ \text{if} \ n\geq2, 
\end{equation}
and for $n=1$:
\begin{equation} \label{w2}
0 \to H_{2} (S_\epsilon - L_0) \to H_1 (F_0) 
 \stackrel{T-1}{\longrightarrow} H_1 (F_0) \to
 H_1 (S_\epsilon - L_0)\to H_0(F_0) \cong \rat \to 0. 
\end{equation}


\section{Deformation Invariance of the Homology of Intersection Spaces}

The main result of this paper asserts that under certain monodromy assumptions (see below), the intersection space homology 
$\redh I_\ast$ for the middle perversity is a smoothing/deformation invariant. Recall that we take homology always with
rational coefficients.
As mentioned in the Introduction, we formally exclude the surface case $n=2$, as a sufficiently general construction
of the intersection space in this case is presently not available, although
the theory $HI^\ast (V;\real)$ has a de Rham description
by global differential forms on the regular part of $V$, \cite{banagl-derhamintspace},
which does not require links to be simply-connected. Using this description
of $HI^\ast$, the theorem can be seen to hold for $n=2$ as well.

\begin{thm}\label{thm1} (Stability Theorem.)
Let $V$ be a complex projective hypersurface of dimension $n\not= 2$ with one isolated
singularity, and let $V_s$ be a nearby smooth deformation of $V$. 
Then, for all $i<2n$ and $i\not= n$, we have
\[ \redh_i (V_s) \cong \redh I_i (V). \]
Moreover, 
\[ H_n (V_s) \cong HI_n (V) \]
if, and only if, the monodromy operator $T$ of the singularity is trivial.
\end{thm}
\begin{proof}
Since $V\cong W$ via a homeomorphism which near the singularity is given by levelwise
extension of a diffeomorphism of the links, 
we may prove the statement for $HI_\ast (W)$ rather than
$HI_\ast (V)$. (See also the proof of Proposition \ref{map} for the construction
of a homeomorphism $IV\cong IW$.)
Suppose $i<n$. Then the exact sequence
\[ 0 = \redh_i (F) \longrightarrow \redh_i (V_s) \longrightarrow
 H_i (V_s,F) \longrightarrow \redh_{i-1} (F)=0 \]
shows that
\[ \redh_i (V_s) \cong H_i (V_s,F). \]
Using the homeomorphism
\begin{equation} \label{equ.wmlvsf}
W \cong \frac{M}{L} = \frac{M\cup_L F}{F} = \frac{V_s}{F},
\end{equation}
we obtain an isomorphism
\[ \redh_i (V_s) \cong \redh_i (W). \]
Since for $i<n,$
\[ \redh I_i (W) \cong   H_i (M,\partial M) \cong \redh_i (W), \]
the statement follows. The case $2n>i>n$ follows from the case $i<n$ by
Poincar\'e duality: If $2n>i>n\geq 0,$ then $0<2n-i<n$ and thus
\[ HI_i (W)\cong HI_{2n-i}(W)^\ast \cong HI_{2n-i}(V_s)^\ast \cong
 H_i (V_s). \]
In degree $i=n$, the $T$-shaped diagrams of Section \ref{background} together with duality and excision yield:
\begin{eqnarray*} HI_n (W) &\cong& H_n(M) \oplus \im (H_n (M,L) \to H_{n-1} L) \\ &\cong& 
H_n (M,L) \oplus \im (H_n L \to H_n M) \\ &\cong&
H_n (W) \oplus \im (H_n L \to H_n M). \end{eqnarray*}
The exact sequence (\ref{equ.specialsequ}) shows that
\[ H_n (V_s) \cong H_n (W) \oplus \im (H_n F \to H_n V_s), \]
whence $H_n (V_s) \cong HI_n (V)$ if, and only if,
\begin{equation} \label{equ.rklmrkfvs}
\rk (H_n L \to H_n M) = \rk (H_n F \to H_n V_s). 
\end{equation}
At this point, we need to distinguish between the cases $n \geq 2$ and $n=1$.\\

Let us first assume that $n \geq 2$. Since $F$ is compact, oriented, and nonsingular, we may use Poincar\'e duality
to deduce $H_{n+1} (F,L)=0$ from $H_{n-1}(F)=0$. The exact sequence of the
pair $(F,L)$,
\[ 0= H_{n+1} (F,L) \stackrel{\partial_\ast}{\longrightarrow}
  H_n (L) \stackrel{j_\ast}{\longrightarrow} H_n (F), \]
implies that $j_\ast: H_n L \to H_n F$ is injective. The inclusion
$j: (M,L) \subset (V_s,F)$ induces an isomorphism
\[ j_\ast: H_{n+1} (M,L) \cong H_{n+1} (V_s,F), \]
by excision, cf. (\ref{equ.wmlvsf}). We obtain a commutative diagram
\[ \xymatrix{
H_{n+1} (V_s,F) \ar[r]^{\partial_\ast} & H_n (F) \\
H_{n+1} (M,L) \ar[u]^{j_\ast}_{\cong} \ar[r]^{\partial_\ast} &
H_n (L) \ar@{^{(}->}[u]_{j_\ast}
} \]
from which we see that
\[ \partial_\ast H_{n+1} (M,L) \cong j_\ast \partial_\ast H_{n+1}(M,L) =
 \partial_\ast j_\ast H_{n+1}(M,L) = \partial_\ast H_{n+1}(V_s,F). \]
Since
\begin{eqnarray*}
\rk (H_n L\to H_n M) & = & \rk H_n L - \rk (\partial_\ast:
  H_{n+1} (M,L)\to H_n L), \\
\rk (H_n F\to H_n V_s) & = & \rk H_n F - \rk (\partial_\ast:
  H_{n+1} (V_s,F)\to H_n F), 
\end{eqnarray*}
equality (\ref{equ.rklmrkfvs}) holds if, and only if,
\begin{equation} \label{equ.hnlhnf}
\rk H_n (L)= \rk H_n (F),
\end{equation}
that is, $\rk H_n (L)=\mu,$ the Milnor number. Using the Alexander duality
isomorphisms
\[ H_{n+1} (S_\epsilon - L) \cong H^{n-1}(L),~
   H_n (S_\epsilon - L) \cong H^n (L) \]
in the Wang sequence (\ref{equ.wang}), we get the exact sequence
\[ 0 \to H^{n-1} (L) \to H_n (F) 
 \stackrel{T-1}{\longrightarrow} H_n (F) \to
 H^n (L)\to 0, \]
which shows that
\[ \rk H_n (L) = \rk H_n F - \rk (T-1). \]
Hence (\ref{equ.hnlhnf}) holds iff $T-1=0$.\\

If $n=1$, the Milnor fiber $F$ is connected, but the link $L$ may have multiple circle components. 
Since $M \cong M_0$ has the homotopy type of a one-dimensional CW complex (as it is homotopic to an affine plane curve), the homology long exact sequence of the pair $(M,L)$ yields that $\partial_\ast :H_{2} (M,L) \to H_1(L)$ is injective. Thus,
\[ \rk (H_1 L\to H_1 M)  =  \rk H_1 L - \rk (\partial_\ast:
  H_{2} (M,L)\to H_1 L) = \rk H_1 L  - \rk H_{2} (M,L). \]
On the other hand, since $F$ has the homotopy type of a bouquet of circles,    
the homology long exact   sequence of the pair $(V_s,F)$ yields that $$\rk \left( i_\ast:  H_2(V_s) \to H_2(V_s,F) \right)=\rk H_2(V_s) =1.$$ Therefore, 
\begin{eqnarray*}
\rk (H_1 F\to H_1 V_s) & = & \rk H_1 F - \rk (\partial_\ast: H_{2} (V_s,F)\to H_1 F) \\
  & = & \rk H_1 F - \rk H_{2} (V_s,F) +1. 
\end{eqnarray*}
Since, by excision, $H_{2} (V_s,F) \cong H_{2} (M,L)$, the
equality (\ref{equ.rklmrkfvs}) holds if, and only if,
\begin{equation}\label{n1}
\rk H_1 (L)= \rk H_1 (F) +1.
\end{equation}
Finally, the Wang exact sequence (\ref{w2}) and Alexander Duality show that   
\[ \rk H_1 (L) = 1+ \rk H_1 F - \rk (T-1). \]
Hence (\ref{n1}) holds iff $T-1=0$.
\end{proof}

 \begin{remark} The only plane curve singularity germ with trivial monodromy operator is a node (i.e., an $A_1$-singularity), e.g., see \cite{O}. In higher dimensions, it is easy to see from the Thom-Sebastiani construction that $A_1$-singularities in an even number of complex variables have trivial monodromy as well.
 \end{remark}
 
 \begin{remark} The algebraic isomorphisms of Theorem \ref{thm1} are obtained abstractly, by computing ranks of the corresponding rational vector spaces. It would be desirable however, to have these algebraic isomorphisms realized by canonical arrows. This fact would then have the following interesting consequences. First, the dual arrows in cohomology (with rational coefficients) would become ring isomorphisms, thus providing non-trivial examples of computations of the internal cup product on the cohomology of an intersection space. Secondly, such canonical arrows would make it possible to import Hodge-theoretic information from the cohomology of the generic fiber $V_s$ onto the cohomology of the intersection space $IV$ associated to the singular fiber.  This program is realized in part in the next section.
 \end{remark}

 \begin{remark}\label{gen} The above theorem can also be formulated in the case of complex projective hypersurfaces with any number of isolated singularities by simply replacing $L$ by $\sqcup_{x \in {\rm Sing}(V)} L_x$ and similarly for the local link complements (where $L_x$ is the link of $x \in {\rm Sing}(V)$), $F$ by $\sqcup_{x \in {\rm Sing}(V)} F_x$ (for $F_x$ the local Milnor fiber at a singular point $x$) and $T$ by $\oplus_{x \in {\rm Sing}(V)}T_x$ (for $T_x:H_n(F_x) \to H_n(F_x)$ the corresponding local monodromy operator). The statement of Theorem \ref{thm1} needs to be modified as follows:
 \begin{itemize}
 \item[(a)] if $n \geq 2$, the changes appear for $i \in \{1, 2n-1 \}$. Indeed, in the notations of Theorem \ref{thm1} we have that: $H_1(M,\partial M) \cong H_1(W) \oplus \rat^{b_0(L)-1}$, so we obtain:
\begin{equation}\label{b1} HI_1(V) \cong H_1(V_s) \oplus  \rat^{b_0(L)-1}, \end{equation}
and similarly for $HI_{2n-1}(V)$, by Poincar\'e duality.
 \item[(b)] if $n=1$, we have an isomorphism $HI_1(V) \cong H_1(V_s)$ if, and only if, $\rk(T-1)=2(r-1)$, where $r$ denotes the number of singular points.
 \end{itemize}
 Indeed, in the cutoff-degree $n$, the $T$-shaped diagram of Section \ref{background} yields (as in the proof of Theorem \ref{thm1}):
\[ HI_n (V) \cong  H_n (M,L) \oplus \im (H_n L \to H_n M), \] where $M$ is obtained from $V$ by removing conical neighborhoods of the singular points. Since, by excision, $H_n(M,L) \cong H_n(V, {\rm Sing}(V))$, the long exact sequence for the reduced homology of the pair $(V, {\rm Sing}(V))$ shows that:
\[ H_n(M,L) \cong \begin{cases} H_n(V), \ \ \ \ \ \ \ \ \ \ \ {\rm if} \ n \geq 2, \\ H_1(V) \oplus \rat^{r-1}, \ {\rm if} \ n=1. \end{cases} \]
So the proof of Theorem \ref{thm1} in the case of the cutoff-degree $n$ applies without change if $n \geq 2$. On the other hand, if $n=1$, we get an isomorphism $HI_1(V) \cong H_1(V_s)$ if, and only if, $$(r-1) + \rk (H_1L \to H_1M) = \rk (H_1F \to H_1V_s).$$ The assertion follows now as in the proof of Theorem \ref{thm1}, by using the identity $$\rk H_1L - \rk H_1F=r-\rk(T-1),$$ which follows from the Wang sequence (\ref{w2}). \hfill$\square$
 
 \end{remark}

The discussion of Remark \ref{gen} also yields the following general result: 
\begin{thm}\label{cor1} Let $V$ be a complex projective $n$-dimensional hypersurface with only isolated
singularities, and let $V_s$ be a  nearby smoothing of $V$. Let  $L$, $\mu$ and $T$ denote the total link, total Milnor number and,  respectively, the total monodromy operator, i.e., $L:=\sqcup_{x \in {\rm Sing}(V)} L_x$, $\mu:=\sum_{x \in {\rm Sing}(V)} \mu_x$, and $T:=\oplus_{x \in {\rm Sing}(V)} T_x$, for $(L_x,\mu_x,T_x)$ the corresponding invariants of an isolated hypersurface singularity germ $(V,x)$. The Betti numbers  of the middle-perversity intersection space $IV$ associated to $V$ are computed as follows:
 \begin{itemize}
 \item[(a)] if $n \geq 2$:
 \[ b_i(IV)=b_i(V_s), \  \ i \notin \{1, n, 2n-1, 2n \} \]
\[ b_1(IV)=b_1(V_s) + b_0(L)-1=b_{2n-1}(IV), \]
 \[ b_n(IV)=b_n(V_s)+b_n(L)-\mu=b_n(V_s) - rk (T-I), \]
\[ b_{2n}(IV)=0. \]
 \item[(b)] if $n=1$:
  \[ b_0(IV)=b_0(V_s)=1, \]
  \[ b_1(IV)=b_1(V_s)+b_1(L)-\mu+r-2=b_1(V_s)-\rk(T-1)+2(r-1),\]
  \[b_2(IV)=0,\]
  where $r$ denotes the number of singular points of the curve $V$.
  \end{itemize}
\end{thm}

As a consequence of Theorem \ref{cor1}, we can now reprove Corollary 2.14 of \cite{banagl-intersectionspaces}
in the special case of a projective hypersurface, using results of \cite{CMS08}.
\begin{cor}\label{prop1} Let $V \subset \mathbb{P}^{n+1}$ be a complex projective  hypersurface with only isolated singularities. The difference between the Euler characteristics of the $\mathbb{Z}$-graded rational vector spaces $\redh I_*(V)$ and $IH_*(V)$ is computed by the formula:
\begin{equation}\label{comp}
\chi(\redh I_*(V)) - \chi(IH_*(V))=-2\chi_{<n}(L),
\end{equation}
where the total link $L$ is the disjoint union of the links of all isolated singularities of $V$, and $\chi_{<n}(L)$ is the truncated Euler characteristic of $L$ defined as $\chi_{<n}(L):=\sum_{i<n} (-1)^i b_i(L)$.
\end{cor}
\begin{proof} For each $x \in {\rm Sing}(V)$, denote by $F_x$, $L_x$ and $\mu_x$ the corresponding Milnor fiber, link and Milnor number, respectively. Note that each link $L_x$ is connected if $n\geq 2$, and  Poincar\'e duality yields: $b_{n-1}(L_x)=b_n(L_x)$.

Let $V_s$ be a nearby smoothing of $V$. Then it is well-known that we have (e.g., see \cite{Dim04}[Ex.6.2.6]):
\begin{equation}\label{one}
\chi(H_*(V_s))-\chi(H_*(V))=\sum_{x \in {\rm Sing}(V)} (-1)^n \mu_x.
\end{equation}
On the other hand,  we get by \cite{CMS08}[Cor.3.5] that:
\begin{equation}\label{two}
\chi(H_*(V))-\chi(IH_*(V))=\sum_{x \in {\rm Sing}(V)} \left(1-\chi(IH_*(c^{\circ}L_x)) \right),
\end{equation}
where $c^{\circ}L_x$ denotes the open cone on the link $L_x$. By using the cone formula for intersection homology with closed supports (e.g., see \cite{Ba07}[Ex.4.1.15]) we note that: 
\begin{equation}\label{ihcone} \chi(IH_*(c^{\circ}L_x))=
\begin{cases}
1+(-1)^{n+1} b_n(L_x),  \ {\rm if} \ n\geq2,\\
b_1(L_x), \ \ \ \ \ \ \ \ \ \ \ \ \ \ \ \ \ {\rm if} \ n=1. 
\end{cases}
\end{equation}
Together with (\ref{two}), this yields:
\begin{equation}\label{three}
\chi(H_*(V))-\chi(IH_*(V))=
\begin{cases}
\sum_{x \in {\rm Sing}(V)} (-1)^n b_n(L_x),  \ {\rm if} \ n\geq2,\\
\sum_{x \in {\rm Sing}(V)}  \left(1-b_1(L_x) \right),  \ {\rm if} \ n=1.
\end{cases}
\end{equation}
Therefore, by combining (\ref{one}) and (\ref{three}), we obtain:
\begin{equation}\label{four}
\chi(H_*(V_s))-\chi(IH_*(V))=
\begin{cases}
\sum_{x \in {\rm Sing}(V)} (-1)^n \left( \mu_x + b_n(L_x) \right) ,  \ {\rm if} \ n\geq2,\\
\sum_{x \in {\rm Sing}(V)} \left( 1-\mu_x-b_1(L_x) \right),  \ \ \ \ {\rm if} \ n=1.
\end{cases}
\end{equation}
Lastly, Theorem \ref{cor1} implies that
\begin{equation}\label{five}
\chi(H_*(V_s))-\chi(\redh I_*(V))=2+2 \sum_{x \in {\rm Sing}(V)} (b_0(L_x)-1) +\sum_{x \in {\rm Sing}(V)} (-1)^n \left( \mu_x - b_n(L_x) \right),
\end{equation}
if $n \geq2$, and 
\begin{equation}\label{six}
\chi(H_*(V_s))-\chi(\redh I_*(V))=r+ \sum_{x \in {\rm Sing}(V)} \left(b_1(L_x)-\mu_x \right),
\end{equation}
if $n=1$, where $r$ denotes the number of singular points of the curve $V$. 
The desired formula follows now by combining the equations (\ref{four}) and (\ref{five}), resp. (\ref{six}), together with Poincar\'e duality for links.
\end{proof}

Let us illustrate our calculations on some simple examples (see also Section \ref{sec.cyquintic} for more elaborate examples involving conifold transitions between Calabi-Yau threefolds).
\begin{example}\label{conic} {\it Degeneration of conics.}\\ Let $V$ be the projective curve defined by 
\[ V:=\{ (x:y:z) \in \mathbb{P}^2 ~|~ yz=0 \}, \] 
that is, a union of two projective lines intersecting at $P=(1:0:0)$.
Topologically, $V$ is an equatorially pinched $2$-sphere, i.e., $S^2 \vee S^2$.
The (join) point $P$ is a nodal singularity, whose link is a union of two circles, the Milnor fiber is a cylinder $S^1 \times I$, and the corresponding monodromy operator is trivial. The associated intersection space $IV$ is 
given by attaching one endpoint of an interval to a northern hemisphere disc and the other endpoint to a 
southern hemisphere disc. Thus $IV$ is contractible and $\redh I_\ast (V)=0$.
It is easy to see (using the genus-degree formula, for example)
that a smoothing $$V_s:=\{ (x:y:z) \in \mathbb{P}^2 ~|~ yz+sx^2=0 \}$$ of $V$ is topologically a sphere $S^2$.
Thus $b_1(IV)=b_1(V_s)=0$. On the other hand, the normalization of $V$ is a disjoint union of two $2$-spheres, so $IH_*(V)=H_*(S^2) \oplus H_*(S^2)$. The formula of Corollary \ref{prop1} is easily seen to be satisfied.
\end{example}
\begin{example} \label{ex.kummer} {\it Kummer surfaces.} \\ 
Let $V$ be a Kummer quartic surface \cite{H05}, i.e., an irreducible algebraic surface of degree $4$ in $\mathbb{P}^3$ with $16$ ordinary double points (this is the maximal possible number of singularities on such a surface). 
The monodromy operator is \emph{not} trivial for this example.
It is a classical fact that a Kummer surface is the quotient of a $2$-dimensional complex torus (in fact, the Jacobian variety of a smooth hyperelliptic curve of genus $2$) by the  involution defined by inversion in the group law. In particular, $V$ is a rational homology manifold. Therefore, $$IH_*(V) \cong H_*(V).$$ 
And it is not hard to see (e.g., cf. \cite{Sp56}) that we have:
\[ H_*(V)=\left(  \rat, 0, \rat^{6}, 0, \rat \right). \]
Each singular point of $V$ has a link homeomorphic to $\mathbb{RP}^3$, and Milnor number equal to $1$ (i.e., the corresponding Milnor fiber is homotopy equivalent to $S^2$). A nearby smoothing $V_s$ of $V$ is a non-singular quartic surface in $\mathbb{P}^3$, hence a $K3$ surface. The Hodge numbers of any smooth $K3$ surface are: $b_{1,0}=0$, $b_{2,0}=1$, $b_{1,1}=20$, thus the Betti numbers of $V_s$ are computed as:
\[ b_0(V_s)=b_4(V_s)=1, \ b_1(V_s)=b_3(V_s)=0, \  b_2(V_s)=22. \]
Let $IV$ be the middle-perversity intersection space associated to the Kummer surface $V$. 
(As pointed out above, there is at present no general construction of the intersection space if the link is not simply
connected. However, 
to construct $IV$ for the Kummer surface, we can use the spatial homology truncation 
$t_{<2} (\real \mathbb{P}^3, Y)=\real \mathbb{P}^2,$
with $Y=C_2 (\real \mathbb{P}^3)=\intg$.)
Then Theorem \ref{cor1} yields that:
\[ b_0(IV)=1, \  b_1(IV)=b_3(IV)=15, \ b_2(IV)=6, \ b_4(IV)=0. \]
And the formula of Corollary \ref{prop1} reads in this case as: $-24-8=-2 \cdot 16$.
We observe that in this example, for the middle degree, all of $H_2,$ $IH_2$ and $HI_2$ agree, but are all different
from $H_2 (V_s)$.
\end{example}


\section{Maps from Intersection Spaces to Smooth Deformations}

The aim of this section is to show that the algebraic isomorphisms of Theorem \ref{thm1} are in most cases induced by continuous maps. 

\begin{prop}\label{map}
Suppose that $V$ is an $n$-dimensional projective hypersurface which has precisely one isolated singularity with link $L$.
If $n=1$ or $n\geq 3$ and $H_{n-1}(L;\intg)$ is torsionfree, then there is a map 
$\eta: IV \to V_s$ such that the diagram
\[ \xymatrix{
M \ar[r] \ar[rd] & IV \ar[r]^{\gamma} \ar[d]_{\eta} & V \\
& V_s \ar[ru]_{r_s} &
} \]
commutes.
\end{prop}
\begin{proof}
Assume $n\geq 3$ and $H_{n-1}(L;\intg)$ torsionfree.
Even without this assumption, the cohomology group
\[ H^{n-1}(L;\intg) = \Hom (H_{n-1}L,\intg)\oplus \Ext (H_{n-2}L,\intg) = 
  \Hom(H_{n-1}L,\intg) \]
is torsionfree. Hence $H_n (L;\intg)\cong H^{n-1}(L;\intg)$ is torsionfree.
Let $\{ z_1, \ldots, z_s \}$ be a basis of $H_n (L;\intg)$.
The link $L$ is $(n-2)$-connected; in particular simply connected, as $n\geq 3$.
Thus minimal cell structure theory applies and yields a cellular homotopy equivalence
$h: \widetilde{L}\to L$, where $\widetilde{L}$ is a CW-complex of the form
\[ \widetilde{L} = \bigvee_{i=1}^r S^{n-1}_i \cup \bigcup_{j=1}^s e^n_j \cup e^{2n-1}.
  \]
The $(n-1)$-spheres $S^{n-1}_i,$ $i=1,\ldots, r,$ generate $H_{n-1}(\widetilde{L};\intg)\cong 
H_{n-1} (L;\intg)=\intg^r$. The $n$-cells
$e^n_j,$ $j=1,\ldots, s$ are cycles, one for each basis element $z_j$. On homology,
$h_\ast$ maps the class of the cycle $e^n_j$ to $z_j$.
Then
\[ L_{<n} := \bigvee_{i=1}^r S^{n-1}_i, \]
together with the map
\[ f: L_{<n} \hookrightarrow \widetilde{L} \stackrel{h}{\longrightarrow} L, \]
is a homological $n$-truncation of the link $L$. We claim that the composition
\[ L_{<n} \stackrel{f}{\longrightarrow} L \hookrightarrow F \]
is nullhomotopic. Let
$b: F\to \bigvee^\mu S^n,$ $b': \bigvee^\mu S^n \to F$ be homotopy inverse homotopy equivalences.
The composition
\[ L_{<n} = \bigvee S^{n-1}_i \stackrel{f}{\longrightarrow} L \to F \stackrel{b}{\longrightarrow}
  \bigvee^\mu S^n \]
is nullhomotopic by the cellular approximation theorem. Thus
\[ L_{<n} \stackrel{f}{\longrightarrow} L \to F \stackrel{b'b}{\longrightarrow} F \]
is nullhomotopic. Since $b'b \simeq \id_F,$ $L_{<n} \to L \to F$ is nullhomotopic, establishing the claim.
Consequently, there exists an extension $\bar{f}:\cone (L_{<n}) \to F$ of 
$L_{<n} \to L \to F$ to the cone. We obtain a commutative diagram
\begin{equation} \label{equ.pushout}
\xymatrix{
\cone (L_{<n}) \ar[d]_{\bar{f}} & L_{<n} \ar[d]_f \ar@{^{(}->}[l] \ar[r]^f & L \ar@{^{(}->}[r] & M \ar@{=}[d] \\
F & L \ar@{^{(}->}[l] \ar@{^{(}->}[rr] & & M.
} \end{equation}
The pushout of the top row is the intersection space $IW$, the pushout of the bottom row
is $V_s$ by construction. Thus, by the universal property of the pushout, the diagram (\ref{equ.pushout})
induces a unique map
\[ IW \longrightarrow V_s \]
such that
\[ \xymatrix{
\cone (L_{<n}) \ar[r] \ar[rd] & IW \ar[d] & M \ar[l] \ar[dl] \\
& V_s &
} \]
commutes. \\

In the curve case $n=1$, the homology $1$-truncation is given by
$L_{<n} = \{ p_1, \ldots, p_l \} \subset L$, where $l = \rk H_0 (L)$ and $p_i$ lies in the $i$-th
connected component of $L$. The map $f: L_{<n}\to L$ is the inclusion of these points.
Let $p\in F$ be a base point. Since $F$ is path connected, we can 
choose paths $I \to F$ connecting each $p_i$ to $p$. These paths define a map
$\overline{f}: \cone (L_{<n})\to F$ such that
\[ \xymatrix{
\cone (L_{<n}) \ar[d]_{\overline{f}} & 
  L_{<n} \ar@{^{(}->}[d]^f \ar@{^{(}->}[l] \ar@{^{(}->}[r]^f & L \ar@{^{(}->}[r] & M \ar@{=}[d] \\
F & L \ar@{^{(}->}[l] \ar@{^{(}->}[rr] & & M.
} \]
commutes. This diagram induces a unique map $IW \to V_s$ as in the case $n\geq 3$. \\

To the end of this proof, we will be using freely the notations introduced in Section \ref{sec.backgroundsing}.
As in the construction of the specialization map $r_s$, we use a collar to write $M_0$ as
$M_0 = \overline{M}_0 \cup [-1,0]\times L_0$ with $\overline{M}_0$ a compact codimension $0$
submanifold of $M_0$, diffeomorphic to $M_0$. The boundary of $\overline{M}_0$ corresponds to
$\{ -1 \} \times L_0$. The diffeomorphism $\psi: M\to M_0$ induces a decomposition
$M = \overline{M} \cup [-1,0] \times L.$ Recall that our model for the cone on a space $A$ is
$\cone (A) = [0,1]\times A / \{ 1 \}\times A.$ To construct $IV$, we may take
$(L_0)_{<n} = L_{<n}$ and we define $f_0: (L_0)_{<n} \to L_0$ to be
\[ (L_0)_{<n} \stackrel{f}{\longrightarrow} L \stackrel{\psi|}{\longrightarrow} L_0. \]
The map
\[ \gamma: IV = \overline{M}_0 \cup [-1,0]\times L_0 \cup \cone (L_0)_{<n} \to
 \overline{M}_0 \cup [-1,0]\times L_0 \cup \cone (L_0) =V \]
maps $\overline{M}_0$ to $\overline{M}_0$ by the identity, 
$\cone (L_0)_{<n}$ to the cone point in $V$, and maps $[-1,0]\times L_0$ by stretching
$[-1,0] \cong [-1,1]$ using the function $h$ from the construction of the specialization map.
Note that then $\gamma|_{[-1,0] \times L_0}$ equals the composition
\[ [-1,0]\times L_0 \stackrel{\psi|^{-1}}{\longrightarrow} [-1,0]\times L
 \stackrel{\rho|}{\longrightarrow} [-1,0]\times L_0 \cup \cone (L_0). \]
Let us introduce the short-hand notation
\[ C = \rho^{-1}_\delta ([-1,0]\times L_0),~
  C_{-1} = \rho^{-1}_\delta (\{ -1 \} \times L_0),~
  C_0 = \rho^{-1}_\delta (\{ 0 \} \times L_0) = B_\epsilon (x_0) \cap N_\delta, \]
\[ \overline{N}_\delta = \rho^{-1}_\delta (\overline{M}_0),~
  D = [-1,0]\times L_0 \cup \cone (L_0). \]
We claim that $\{ -1 \} \times L \subset C_{-1}$. The claim is equivalent
to $\rho_\delta (\{ -1 \} \times L)\subset \{ -1 \}\times L_0$, which follows from
$\rho_\delta (\{ -1 \} \times L) = \psi (\{ -1 \} \times L)$ and the fact that the collar
on $M$ has been constructed by composing the collar on $M_0$ with the inverse of $\psi$.
Similarly,
\[ \{ 0 \} \times L \subset C_{0},~ [-1,0] \times L \subset C. \]
The commutative diagram
\[ \xymatrix{
\overline{M}_0 \ar[d]^{\psi|^{-1}} & \{ -1 \} \times L_0 \ar[d]^{\psi|^{-1}} \ar@{^{(}->}[l] \ar@{^{(}->}[r] 
 & [-1,0] \times L_0 \ar[d]^{\psi|^{-1}} & 
  \{ 0 \}\times (L_0)_{<n} \ar@{=}[d]\ar@{^{(}->}[l]_{f_0} \ar@{^{(}->}[r] &
  \cone (L_0)_{<n}  \ar@{=}[d] \\
\overline{M} & \{ -1 \} \times L \ar@{^{(}->}[l] \ar@{^{(}->}[r] & [-1,0] \times L & 
  \{ 0 \}\times L_{<n} \ar@{^{(}->}[l]_f \ar@{^{(}->}[r] &
  \cone L_{<n}  
} \]
induces uniquely a homeomorphism $IV \stackrel{\cong}{\longrightarrow} IW$, as
$IV$ is the colimit of the top row and $IW$ is the colimit of the bottom row.
The map $\eta: IV\to V_s$ is defined to be the composition
\[ IV \stackrel{\cong}{\longrightarrow} IW \longrightarrow V_s. \]
We analyze the composition
\[ IV \stackrel{\cong}{\longrightarrow} IW \longrightarrow V_s \hookrightarrow N
  \stackrel{\rho}{\longrightarrow} V \]
of $\eta$ with the specialization map
by considering the commutative diagram
\[ \xymatrix{
\overline{M}_0 \ar[r]^{\psi|^{-1}} & \overline{M} \ar@{=}[r] &
  \overline{M} \ar@{^{(}->}[r] & \overline{N}_\delta \ar[r]^{\rho_\delta|} & 
  \overline{M}_0 \\
\{ -1 \}\times L_0 \ar@{^{(}->}[u] \ar@{^{(}->}[d] \ar[r]^{\psi|^{-1}} & 
  \{ -1 \} \times L \ar@{^{(}->}[u] \ar@{^{(}->}[d] \ar@{=}[r] &
  \{ -1 \} \times L \ar@{^{(}->}[u] \ar@{^{(}->}[d] \ar@{^{(}->}[r] & 
  C_{-1} \ar@{^{(}->}[u] \ar@{^{(}->}[d] \ar[r]^{\rho_\delta|} & 
  \{ -1 \} \times L_0 \ar@{^{(}->}[u] \ar@{^{(}->}[dd] \\
[-1,0] \times L_0 \ar[r]^{\psi|^{-1}} & [-1,0] \times L \ar@{=}[r] &
  [-1,0] \times L \ar@{^{(}->}[r] & C \ar[rd]^{\rho|} & \\
\{ 0 \}\times (L_0)_{<n} \ar[u]^{f_0} \ar@{^{(}->}[d] \ar@{=}[r] & 
  \{ 0 \} \times L_{<n} \ar[u]^f \ar@{^{(}->}[d] \ar[r]^f &
  \{ 0 \} \times L \ar@{^{(}->}[u] \ar@{^{(}->}[d] \ar@{^{(}->}[r] & 
 C_{0} \ar@{^{(}->}[u] \ar@{^{(}->}[d] \ar[r]^{\operatorname{const}} &  D \\
\cone (L_0)_{<n} \ar@{=}[r] & \cone L_{<n} \ar[r]^{\overline{f}} &
  F \ar@{^{(}->}[r] & B\cup C_0 \ar[ru]_{\operatorname{const}} &   \\
} \]
The colimits of the columns are, from left to right, $IV,$ $IW,$ $V_s,$ $N$ and $V$.
Since $\overline{M} \hookrightarrow \overline{N}_\delta 
\stackrel{\rho_\delta}{\longrightarrow} \overline{M}_0$ is
$\psi|: \overline{M}\to \overline{M}_0$, etc., we see that the composition from the
leftmost column to the rightmost column is given by
\[ \xymatrix{
\overline{M}_0 \ar[r]^{\psi| \circ \psi|^{-1} =\id} & \overline{M}_0 \\
\{ -1 \} \times L_0 \ar@{^{(}->}[u] \ar@{^{(}->}[d] \ar[r]^{id} & \{ -1 \} \times L_0 \ar@{^{(}->}[u] \ar@{^{(}->}[dd] \\
[-1,0] \times L_0 \ar[rd]^{\gamma| = \rho| \circ \psi|^{-1}} & \\
\{ 0 \} \times (L_0)_{<n} \ar[u]^{f_0} \ar@{^{(}->}[d] \ar[r]^{\operatorname{const}} & D \\
\cone (L_0)_{<n}, \ar[ru]_{\operatorname{const}} &
} \]
which is $\gamma$.
\end{proof}
The torsion freeness assumption in the above proposition can be eliminated as long as the link is still
simply connected. Indeed, since in the present paper we are only interested in rational homology,
it would suffice to construct a rational model $IV_\rat$ of the intersection space $IV$. This can be done
using Bousfield-Kan localization and the odd-primary spatial homology truncation developed in
Section 1.7 of \cite{banagl-intersectionspaces}. For example, if the link of a surface singularity
is a rational homology 3-sphere $\Sigma$, then the rational spatial homology $2$-truncation
$t^\rat_{<2} \Sigma$ is a point. The reason why we exclude the surface case in the above 
proposition is that surface links are generally not simply connected, which general spatial
homology truncation requires. This does not preclude the possibility of constructing a map
$IV \to V_s$ for a given surface $V$ by using ad-hoc devices. For example, if one has an
ADE-singularity, then the link is $S^3 /G$ for a finite group $G$ and $\pi_1 (S^3 /G)=G$.
Thus the link is a rational homology $3$-sphere and the rational $2$-truncation is a point.\\

We can now prove the following result:
\begin{thm} \label{isomap}
Let $V \subset \mathbb{CP}^{n+1}$ be a complex projective hypersurface with precisely one isolated singularity and with link $L$.
If $n=1$ or $n\geq 3$ and $H_{n-1}(L;\intg)$ is torsionfree, then the algebraic isomorphisms of Theorem \ref{thm1} 
can be taken to be induced by the  map 
$\eta: IV \to V_s$ constructed in Proposition \ref{map}. In particular, the dual isomorphisms in cohomology are ring isomorphisms.
\end{thm}
\begin{proof} 
The map $\eta: IV\to V_s$ is a composition
\[ IV \stackrel{\cong}{\longrightarrow} IW \stackrel{\eta'}{\longrightarrow} V_s. \]
As the first map is a homeomorphism, it suffices to show that $\eta'$ is a rational homology isomorphism.

We begin by considering the following diagram of long exact Mayer-Vietoris
sequences for the commutative diagram (\ref{equ.pushout}) of pushouts (e.g., see \cite{GM81}[19.5]):
\begin{equation} \label{dp}
\xymatrix{
\cdots \ar[r] & \redh_i(L_{<n}) \ar[d]_{f_*}\ar[r]  & \redh_i (\cone (L_{<n})) \oplus \redh_i(M) \ar[r] \ar[d]_{(\bar{f}_*,id)} & \redh_i(IW) \ar[r] \ar[d]_{\eta'_*} & \redh_{i-1}(L_{<n}) \ar[r] \ar[d]_{f_*} & \cdots \\
\cdots \ar[r] & \redh_i(L) \ar[r] & \redh_i(F) \oplus \redh_i(M) \ar[r] & \redh_i(V_s) \ar[r] & \redh_{i-1}(L) \ar[r] & \cdots 
} \end{equation}
Recall that, by construction, $f_*: \redh_i(L_{<n}) \to \redh_i(L)$ is an isomorphism if $i < n$, and $\redh_i(L_{<n}) \cong 0$ if $i \geq n$. Also, as stated in Section \ref{sec.backgroundsing}, $L$ is $(n-2)$-connected, and $F$ is $(n-1)$-connected. 

Let us first assume that $n \geq 3$. Then if $i <n$, a five-lemma argument on the diagram (\ref{dp}) yields that $\eta'_{\ast}:\redh_i(IW) \to \redh_i(V_s)$ is an isomorphism. If $i \geq n$, recall from the proof of Theorem \ref{thm1} that $j_*:H_n(L) \to H_n(F)$ is injective. In particular, the map $H_{n+1}(V_s) \to H_n(L)$ is the zero homomorphism. Then diagram (\ref{dp}) yields that $\eta'_{\ast}:\redh_i(IW) \to \redh_i(V_s)$ is an isomorphism for $n<i<2n$.
For $i=n$, there is a commutative diagram:

\begin{equation} \label{dp2}
\xymatrix{
0 \ar[r] & 0 \ar[d]\ar[r]  & 0 \oplus \redh_n(M) \ar[r] \ar[d]_{(\bar{f}_*,id)} & \redh_n(IW) \ar[r] \ar[d]_{\eta'_*} & \redh_{n-1}(L_{<n}) \ar[r] \ar[d]_{f_*}^{\cong} &\\
0 \ar[r] & \redh_n(L) \ar[r] & \redh_n(F) \oplus \redh_n(M) \ar[r] & \redh_n(V_s) \ar[r] & \redh_{n-1}(L) \ar[r] &
} \end{equation}
\begin{equation*}
\xymatrix{
\ar[r] & \redh_{n-1}(M) \ar[r] \ar[d]_{id}^{\cong} & \redh_{n-1}(IW) \ar[r] \ar[d]_{\eta'_*}^{\cong} & 0\\
\ar[r] & \redh_{n-1}(M) \ar[r] &  \redh_{n-1}(V_s) \ar[r]  & 0.
} \end{equation*}
By excision, Poincar\'e duality and the connectivity of $F$,
\[ H_{n+1}(V_s, M)\cong H_{n+1} (F,L) \cong H^{n-1} (F)=0. \]
Thus the exact sequence of the pair $(V_s, M),$
\[ 0=H_{n+1}(V_s,M) \longrightarrow H_n (M) \longrightarrow H_n (V_s), \]
shows that the map $\iota_{s\ast}: H_n (M)\to H_n (V_s)$ is injective.
Let $x\in H_n (IW)$ be an element such that $\eta'_\ast (x)=0$. Then, as $f_\ast$ is an isomorphism
in degree $n-1$, $x=\iota_\ast (m)$ for some $m\in H_n (M),$ where $\iota_\ast$ denotes the map
$\iota_\ast: H_n (M)\to H_n (IW)$. Thus $\iota_{s\ast} (m)=0$ and, by the injectivity of 
$\iota_{s\ast},$ $m=0$. It follows that $\eta'_\ast$ is a monomorphism (whether or not the monodromy is trivial).
Therefore, $\eta'_\ast$ is an isomorphism iff $\rk H_n (IW) = \rk H_n (V_s)$.
By the Stability Theorem \ref{thm1}, this is equivalent to $T$ being trivial. \\

If $n=1$, recall that $H_2(M) \cong 0$, $H_2(F) \cong 0$, $H_2(IV)=0$ and $H_2(V_s) \cong \rat$. So the relevant part of diagram (\ref{dp}) is:
\begin{equation} \label{dp3}
\xymatrix{
0 \ar[r] & 0 \ar[d]\ar[r] & 0 \ar[d]\ar[r]  & 0 \oplus \redh_1(M) \ar[r] \ar[d]_{(\bar{f}_*,id)} & \redh_1(IW) \ar[r] \ar[d]_{\eta'_*} & \redh_{0}(L_{<1}) \ar[r] \ar[d]_{f_*}^{\cong} &  0\\
0 \ar[r] & \rat \ar[r] & \redh_1(L) \ar[r] & \redh_1(F) \oplus \redh_1(M) \ar[r] & \redh_1(V_s) \ar[r] & \redh_{0}(L) \ar[r] &  0.
} \end{equation}
The group $H_2 (F,L)\cong \rat$ is generated by the fundamental class $[F,L]$.
The connecting homomorphism $H_2 (F,L)\to H_1 (L)$ maps $[F,L]$ to the fundamental
class $[L]$ of $L$. As $L=\partial M,$ the image of $[L]$ under 
$H_1 (L)\to H_1 (M)$ vanishes. This, together with the commutative square
\[ \xymatrix{
H_2 (V_s, M) \ar[r] & H_1 (M) \\
H_2 (F,L) \ar[u]_{\cong} \ar[r] & H_1 (L) \ar[u]
} \]
shows that $H_2 (V_s,M)\to H_1 (M)$ is the zero map. Consequently, 
$H_1 (M)\to H_1 (V_s)$ is a monomorphism, as in the case $n\geq 3$.
It follows as above that $\eta'_\ast$ is injective. The claim then follows from Theorem \ref{thm1}.
\end{proof}
The above proof shows that $\eta_\ast: HI_n (V)\to H_n (V_s)$ is always injective. The assumption on the monodromy
is needed for surjectivity. Hence we obtain the two-sided bound
\[   \max \{ \rk IH_n (V), \rk H_n (M), \rk H_n (M,\partial M) \}   \leq \rk HI_n (V) \leq \rk H_n (V_s) \]
in the middle degree. These inequalities are sharp: By the Stability Theorem \ref{thm1}, the upper bound is attained
for trivial monodromy $T$, and the lower bound is attained for the Kummer surface of Example \ref{ex.kummer}.
The bounds show that regardless of the monodromy assumption of the Stability Theorem, 
$HI_\ast (V)$ is generally a better approximation
of $H_* (V_s)$ than intersection homology or ordinary homology.

\begin{cor} \label{cor.hodgestruct}
Under the hypotheses of Theorem \ref{isomap}, let us assume moreover that the local monodromy operator associated to the singularity of $V$ is trivial. Then the rational cohomology groups $HI^i(V) $ of the intersection space can be endowed with rational mixed Hodge structures, so that the canonical map $\gamma:IV \to V$ induces homomorphisms of mixed Hodge structures in  cohomology.
\end{cor}
\begin{proof} By Proposition \ref{map},  there is a  map
$\eta: IV \to V_s$ so that $\gamma:IV \to V$ is the composition $IV \overset{\eta}{\to} V_s \overset{r_s}{\to} V$. Then
$\gamma^*:H^*(V) \to H^*(IV)$ can be factored as $\gamma^*=\eta^* \circ r^*_s$. Moreover, by classical Hodge theory,
$r^*_s:H^*(V) \to H^*(V_s)$ is a mixed Hodge structure homomorphism, where $H^*(V_s)$ carries the ``limit mixed Hodge structure" (cf. \cite{Sc,Ste}, but see also  \cite{PS}[Sect.11.2]). Finally, Theorem \ref{isomap} yields that $\eta^*:H^*(V_s) \to H^*(IV)$ is an isomorphism of rational vector spaces. Therefore, $H^*(IV)$ inherits a rational mixed Hodge structure via $\eta^*$, i.e., the limit mixed Hodge structure, and the claim follows.
\end{proof}

Before discussing examples, let us say a few words about the limit mixed Hodge structure on $H^*(V_s)$. Consider the restriction $\pi|: X^\ast \to S^\ast$ of the projection $\pi$ to the punctured disc $S^\ast$. The loop winding once counterclockwise around the origin gives a generator of $\pi_1(S^\ast, s)$, $s \in S^\ast$. Its action on the fiber $V_s:= \pi|^{-1}(s)$ is well-defined up to homotopy, and it defines on $H^k_{\infty}:=H^k(V_s)$ ($k \in \intg$) the monodromy automorphism $M$. The operator $M:H^k_{\infty} \to H^k_{\infty}$ is quasi-unipotent, with nilpotence index $k$, i.e., there is $m>0$ so that $(M^m-I)^{k+1}=0$. Let $N:=\log M_u$ be the logarithm of the unipotent part in the Jordan decomposition of $M$, so $N$ is a nilpotent operator (with $N^{k+1}=0$). The (monodromy) weight filtration $W^{\infty}$ of the limit mixed Hodge structure is the unique increasing filtration  on $H^k_{\infty}$ such that $N(W^{\infty}_j) \subset W^{\infty}_{j-2}$ and $N^j:Gr^{W^{\infty}}_{k+j} H^k_{\infty} \to Gr^{W^{\infty}}_{k-j} H^k_{\infty}$ is an isomorphism for all $j \geq 0$. The Hodge filtration $F_{\infty}$ of the limit mixed Hodge structure is constructed in \cite{Sc} as a limit (in a certain sense) of the Hodge filtrations on nearby smooth fibers, and in \cite{Ste} by using the relative logarithmic de Rham complex. It follows that for $V_s$ a smooth fiber of $\pi$, we have the equality: ${\rm dim} F^pH^k(V_s)={\rm dim} F_{\infty}^pH^k_{\infty}$. We finally note that the semisimple part $M_s$ of the monodromy is an automorphism of mixed Hodge structures on $H^k_{\infty}$. Also, $N:H^k_{\infty} \to H^k_{\infty}$ is a morphism of mixed Hodge structures of weight $-2$.

We now discuss the case of curve degenerations  satisfying the assumptions of Corollary \ref{cor.hodgestruct}, i.e., the singular fiber of the family has a nodal singularity. Let $\pi:X \to S$ be a degeneration of curves of genus $g$, i.e., $C_s:=
\pi^{-1}(s)$ is a smooth complex projective curve for $s \neq 0$, with first betti number $b_1(C_s)=2g$, and assume that the special fiber $C_0$ has only one singularity which is a node. For the limit mixed Hodge structure on $H^1_{\infty}:=H^1(C_s)$ (or, equivalently, the mixed Hodge structure on $HI^1(C_0)$), we have the monodromy weight filtration: $$H^1_{\infty}=W^{\infty}_2 \supset W^{\infty}_1 \supset W^{\infty}_0.$$ On $W^{\infty}_0$ there is a pure Hodge structure of weight $0$ and type $(0,0)$. Since $N$ is a morphism of weight $-2$ and $N:W^{\infty}_2/ W^{\infty}_1 \overset{\sim}{\to} W^{\infty}_0$ is an isomorphism,  it follows that $W^{\infty}_2/ W^{\infty}_1$ is a pure Hodge structure of weight $2$ and type $(1,1)$. Note that the monodromy $M$ is trivial (or, equivalently, $N=0$) on $H^1(C_s)$ if and only if $W^{\infty}_0=0$, and in this case the monodromy weight filtration is the trivial one, i.e., $H^1_{\infty}=W^{\infty}_1 \supset 0$.
\begin{example} Consider a family of smooth genus $2$ curves $C_s$ degenerating into a union of two smooth elliptic curves meeting transversally at one double point $P$. Write $C_0=E_1 \cup E_2$ for the singular fiber of the family. A Mayer-Vietoris argument shows that $H^1(C_0) \cong H^1(E_1) \oplus H^1(E_2)$, so $H^1(C_0)$ carries a pure Hodge structure of weight $1$. For the limit mixed Hodge structure on $H^1(C_s)$, the Clemens-Schmid exact sequence yields that $$W^{\infty}_1 \cong W_1 H^1(C_0) \cong H^1(C_0) \ \ {\rm and} \ \ W^{\infty}_0 \cong W_0 H^1(C_0) \cong 0.$$ Thus the monodromy representation $M$ is trivial, i.e., $N=0$.
\end{example}
\begin{example} Consider the following family of plane curves
$$y^2=x(x-a_1)(x-a_2)(x-a_3)(x-a_4)(x-s)$$
(or its projectivization in $\cplx P^2$), where the $a_i$'s are distinct non-zero complex numbers. For $s \neq 0$ small enough, the equation defines a Riemann surface (or complex projective curve) $C_s$ of genus $2$. The singular fiber $C_0$ is an elliptic curve with a node. Its normalization $\widetilde{C}_0$ is a smooth elliptic curve.  If $\{\delta_1, \delta_2\}$ and, resp., $\{\gamma_1, \gamma_2\}$ denote the two meridians and, resp., longitudes generating $H_1(C_s;\intg)$, the degeneration can be seen geometrically as contracting the meridian (vanishing cycle) $\delta_1$ to a point. Let us denote by $\{\delta^i, \gamma^i\}_{i=1,2}$ the basis of $H^1(C_s;\intg)$ dual to the above homology basis.
If $p:\widetilde{C}_0 \to C_0$ denotes the normalization map, it is easy to see that $p^*:H^1(C_0) \to H^1(\widetilde{C}_0)$ is onto, with kernel generated by $\{\gamma^1\}$. The cohomology group $H^1(C_0)$ carries a canonical mixed Hodge structure with weight filtration defined by: $$W_0={\rm Ker} (p^*)  \ \ {\rm and} \ \ W_1=H^1(C_0).$$ The limit mixed Hodge structure on $H^1(C_s)$, i.e., the mixed Hodge structure on $HI^1(C_0)$, has weights $0$, $1$ and $2$, with the monodromy weight filtration defined by: $$W^{\infty}_0=\{\gamma^1\}, \ W^{\infty}_1=W^{\infty}_0  \oplus \rat \{ \delta^2, \gamma^2\}, \ W^{\infty}_2=W^{\infty}_1 \oplus \rat \{ \delta^1 \},$$ or in more intrinsic terms: $$W^{\infty}_0={\rm Image} (N), \ W^{\infty}_1={\rm Ker} (N).$$
Note that  $W^{\infty}_1 \cong H^1(C_0)$, so the mixed Hodge structure on $HI^1(C_0)$ determines the mixed Hodge structure of $H^1(C_0)$.
\end{example}


\section{Deformation of Singularities and Intersection Homology}\label{deform}

In this section we investigate deformation properties of intersection homology groups.
As in Section \ref{sec.backgroundsing}, let $\pi: X\to S$ be a smooth deformation of the singular hypersurface
\[ V = \pi^{-1}(0)=V(f) = \{ x\in \mathbb{P}^{n+1} ~|~ f(x)=0 \} \] with only isolated singularities $p_1, \cdots, p_r$.
We consider the nearby and vanishing cycle complexes associated to $\pi$ as follows (e.g., see \cite{De} or \cite{Dim04}[Sect.4.2]). 
Let $\hbar$ be the
complex upper-half plane (i.e., the universal cover of the punctured disc $S^\ast$ via the map $z \mapsto \exp(2\pi i z)$). 
With $X^\ast = X-V,$ the projection $\pi$ restricts to $\pi|: X^\ast \to S^\ast.$ The \emph{canonical fiber} $V_\infty$ of
$\pi$ is defined by the cartesian diagram
\[ \xymatrix{
V_\infty \ar[r] \ar[d] & X^\ast \ar[d]^{\pi|} \\
\hbar \ar[r] & S^\ast.
} \]
Let $k:V_{\infty} \to X^\ast \hookrightarrow X$ be the composition of the induced map with the inclusion, and denote by $i: V=V_0 \hookrightarrow X$ the inclusion of the singular fiber. 
Then the \emph{nearby cycle complex} is the bounded constructible sheaf complex defined by
\begin{equation} \psi_{\pi}(\rat_X):=i^*Rk_*k^* \rat_X \in D^b_c(V).\end{equation}  
If $r_s:V_s \to V$ denotes the specialization map, then by using a resolution of singularities it can be shown that $\psi_{\pi}(\rat_X) \simeq R{r_s}_*\rat_{V_s}$ (e.g., see \cite{PS}[Sect.11.2.3]). 
The \emph{vanishing cycle complex} $\phi_{\pi}(\rat_X) \in D_c^b(V)$ is the
cone on the comparison morphism $\rat_{V}=i^*\rat_X \to
\psi_{\pi}(\rat_X)$ induced by adjunction, i.e., there exists a canonical morphism
$can:\psi_{\pi}(\rat_X) \to \phi_{\pi}(\rat_X)$ such that
\begin{equation}\label{sp}
i^*\rat_X \to \psi_{\pi}(\rat_X) \overset{can}{\to} \phi_{\pi}(\rat_X)
\overset{[1]}{\to}\end{equation} is a distinguished triangle in
$D^b_c(V)$. In fact, by replacing $\rat_X$ by any complex in $D^b_c(X)$, we obtain
in this way functors \[ \psi_{\pi}, \phi_{\pi}:D^b_c(X) \to D^b_c(V). \] 

It follows directly from the definition that for any $x \in V=V_0$, 
\begin{equation}\label{Mi}
H^j(F_x)=\Ha^j(\psi_{\pi} \rat_X)_x \ \ \ {\rm and} \ \ \ 
\redh^j(F_x)=\Ha^j(\phi_{\pi} \rat_X)_x,\end{equation}
where $F_x$ denotes the (closed) Milnor fiber  of $\pi$ at $x$. Since $X$ is smooth,  the identification in (\ref{Mi}) can be used to show that \[ \text{Supp}
(\phi_{\pi} \rat_X) \subseteq \text{Sing}(V). \] And in fact the two sets are identified, e.g., \cite{Dim04}[Cor.6.1.18]. Moreover, since $V$ has only isolated singularities, and the germ $(V,x)$ of such a singularity is identified as above with the germ of $\pi$ at $x$, $F_x$ is in fact the Milnor fiber of the isolated hypersurface singularity germ $(V,x)$.

By applying the hypercohomology functor to the distinguished triangle (\ref{sp}), we get by (\ref{Mi}) the  long exact sequence:
\[  \cdots \to H^j(V) \to H^j(V_s) \to \oplus_{i=1}^r \redh^j(F_{p_i})  \to H^{j+1}(V) \to \cdots  \] Since $p_i$ ($i=1,..,r$) are isolated singularities, this further yields that 
\[ H^j(V) \cong  H^j(V_s) \ \ \text{for} \ j \neq \{n, n+1\} , \]
together with the exact specialization sequence dual to (\ref{equ.specialsequ}):
\begin{multline}\label{spcoh}  0 \to H^n(V) \to H^n(V_s) \to \oplus_{i=1}^r \redh^n(F_{p_i})  \to H^{n+1}(V) \to 
H^{n+1}(V_s) \to 0. \end{multline}

Let us now consider the sheaf complex \[ \Fa:=\psi_{\pi} \rat_X[n] \in D^b_c(V).\]
Since $X$ is smooth and $(n+1)$-dimensional, it is known that $\Fa$ is a perverse self-dual complex on $V$, and we get by (\ref{Mi}) that 
\[  \Fa \vert_{V_{\rm reg}} \simeq \rat_{V_{\rm reg}}[n], \] 
where $V_{\rm reg}:=V \setminus \{p_1, \cdots, p_r \}$ denotes the smooth locus of the hypersurface $V$. Since $\rat_{V_{\rm reg}}[n]$ is a perverse sheaf on $V_{\rm reg}$, we note that $\Fa$ is a perverse (self-dual) extension of $\rat_{V_{\rm reg}}[n]$ to all of $V$. However, the simplest such perverse (self-dual) extension is the (middle-perversity) {\it intersection cohomology complex} 
\begin{equation}\label{ic} IC_V:= \tau_{\leq -1} (Rj_* \rat_{V_{\rm reg}}[n]), \end{equation} with $j: V_{\rm reg} \hookrightarrow V$ the inclusion of the regular part, and $\tau_{\leq}$ the natural truncation functor on $D^b_c(V)$. (Recall that we work under the assumption that the complex projective hypersurface $V$ has (at most) isolated singularities.) Since the hypercohomology of $\Fa$ calculates the rational cohomology $H^*(V_s)$ of a smooth deformation $V_s$ of $V$, and the hypercohomology of $IC_V$ calculates the intersection cohomology $IH^*(V)$ of $V$, it is therefore natural to try to understand the relationship between the sheaf complexes $\Fa$ and $IC_V$. 
 
We have the following:
\begin{prop}\label{thm2} There is a quasi-isomorphism of sheaf complexes $\Fa \simeq IC_V$ if, and only if, the hypersurface $V$ is non-singular. If this is the case, then: \[ H^*(V_s) \cong IH^*(V) \cong H^*(V). \]
\end{prop}
\begin{proof} The ``if" part of the statement follows from the distinguished triangle (\ref{sp}), since for $V$ smooth we have $\text{Supp} (\phi_{\pi} \rat_X) = \text{Sing}(V)=\emptyset$ (cf. \cite{Dim04}[Cor.6.1.18]) and $IC_V=\rat_V[n]$.

Let us now assume that there is a quasi-isomorphism $\Fa \simeq IC_V$. 
Then for any $x \in V$ and $j \in \intg$, there is an isomorphism of rational vector spaces:
\begin{equation}\label{stalk} \Ha^j(\Fa)_x \cong \Ha^j(IC_V)_x. \end{equation}
Assume, moreover, that there is a point $x \in V$ which is an isolated singularity of $V$ (i.e., if $g:(\mbc^{n+1},0) \to (\mbc,0)$ is an analytic function germ representative for $(V,x)$ then $dg(0)=0$). Then if  $F_x$ denotes the corresponding Milnor fiber, the Lefschetz number $\Lambda(h)$ of the monodromy homeomorphism $h:F_x \to F_x$ must vanish (e.g., see \cite{Dim04}[Cor.6.1.16]). So the Milnor fiber $F_x$ must satisfy $\redh^{\ast}(F_x) \neq 0$ (otherwise, $\Lambda(h)=1$), or equivalently, $H^n(F_x) \neq 0$. On the other hand, the identities (\ref{Mi}) and  (\ref{stalk}) yield:
\[  H^n(F_x) \cong \Ha^0(\Fa)_x \cong \Ha^0(IC_V)_x =0, \] where the last vanishing follows from the definition (\ref{ic}) of the complex $IC_V$. We therefore get a contradiction.
\end{proof}

\begin{remark}If the hypersurface $V$ is singular (i.e., the points $p_i$ are indeed singularities), the precise relationship between the two complexes $\Fa$ and $IC_V$ is in general very intricate. However, some information can be derived if one considers these two complexes as elements in Saito's category $\mh(V)$ of mixed Hodge modules on $V$. More precisely, $IC_V$ is a direct summand of $Gr^W_{n}\Fa$, where $W$ is the {\it weight filtration} on $\Fa$ in $\mh(V)$ (compare \cite{Sa}[p.152-153], \cite{CMSS}[Sect.3.4]).\end{remark}

\begin{remark} As Proposition \ref{thm2} suggests, intersection homology is not a smoothing invariant. On the other hand, it is known that intersection homology is invariant under {\it small resolutions}, i.e., if $\ensuremath{\widetilde{V}} \to V$ is a small resolution of the complex algebraic variety $V$ (provided such a resolution exists), then we have isomorphisms
\[ IH^*(V) \cong IH^*(\ensuremath{\widetilde{V}}) \cong H^*(\ensuremath{\widetilde{V}}). \]
Therefore, as suggested by the conifold transition picture (see \cite{banagl-intersectionspaces}[Ch.3]), the trivial monodromy condition arising in Theorem \ref{thm1} can be thought as being mirror symmetric to the condition of existence of a small resolution. More generally, the result of Theorem \ref{isomap} on the injectivity of the map $\eta_\ast: HI_* (V)\to H_* (V_s)$ ``mirrors" the well-known fact that the intersection homology of a complex variety $V$ is a vector subspace of the ordinary homology of any resolution of $V$ (the latter being an easy application of the Bernstein-Beilinson-Deligne-Gabber decomposition theorem).
\end{remark}


\section{Higher-Dimensional Examples: Conifold Transitions}
\label{sec.cyquintic}

We shall illustrate our results on examples derived from the study of conifold transitions (e.g., see \cite{Ro06,banagl-intersectionspaces}.

\begin{example} Consider 
the quintic
\[ P_s (z)= z_0^5 + z_1^5 + z_2^5 + z_3^5 + z_4^5 - 
   5(1+s) z_0 z_1 z_2 z_3 z_4, \]
depending on a complex structure parameter $s$. The variety
\[ V_s = \{ z \in \mathbb{P}^4 ~|~ P_s (z)=0 \} \]
is Calabi-Yau. It is smooth for small $s \not= 0$ and
becomes singular for $s = 0$. 
(For $V_s$ to be singular, $1+s$ must be fifth root of
unity, so $V_s$ is smooth for $0 < |s| < 
| e^{2\pi i/5} -1|$.) 
We write $V = V_0$ for the singular variety.
Any smooth quintic hypersurface
in $\mathbb{P}^4$ (is Calabi-Yau and) has Hodge numbers $b_{1,1} =1$
and $b_{2,1} = 101.$ Thus for $s \not= 0,$
\[ b_2 (V_s) = b_{1,1} (V_s)=1,~ b_3 (V_s) = 2(1+b_{2,1})=204. \]
The singularities are those points
where the gradient of $P_0$ vanishes. If one of the five homogeneous
coordinates $z_0, \ldots, z_4$ vanishes, then the gradient equations
imply that all the others must vanish, too. This is not a point on
$\mathbb{P}^4$, and so all coordinates of a singularity must be nonzero. 
We may then normalize the first one to be $z_0 =1$. From the gradient
equation $z_0^4 = z_1 z_2 z_3 z_4$ it follows that $z_1$ is 
determined by the last three coordinates, $z_1 = (z_2 z_3 z_4)^{-1}$.
The gradient equations also imply that
\[ 1 = z_0^5 = z_0 z_1 z_2 z_3 z_4 = z_1^5 = z_2^5 = z_3^5 = z_4^5, \]
so that all coordinates of a singularity are fifth roots of unity.
Let $(\omega, \xi, \eta)$ be any triple of fifth roots of unity.
(There are $125$ distinct such triples.) The $125$ points
\[ (1 : (\omega \xi \eta)^{-1} : \omega : \xi : \eta ) \]
lie on $V_0$ and the gradient vanishes there. These are thus the
$125$ singularities of $V_0$. 
Each one of them is a node, whose
neighborhood therefore looks topologically like the cone on the $5$-manifold
$S^2 \times S^3$. By replacing each node with a $\mathbb{P}^1,$ one obtains
a small resolution $\ensuremath{\widetilde{V}}$ of $V$.
By \cite{schoenc}, $b_2 (\ensuremath{\widetilde{V}}) = b_{1,1} (\ensuremath{\widetilde{V}})=25$ for the small resolution $\ensuremath{\widetilde{V}}$.
The intersection homology of a singular space is isomorphic to the ordinary
homology of any small resolution of that space. Thus $IH_\ast (V) \cong
H_\ast (\ensuremath{\widetilde{V}})$. Using the information summarized so far, one calculates the following
ranks ($s \not= 0$):
\[ \begin{array}{|c|c|c|c|} \hline
i & \rk H_i (V_s) & \rk H_i (V) & \rk IH_i (V) \\ \hline \hline
2 & 1     & 1   & 25 \\ \hline
3 & 204 & 103 & 2 \\ \hline
4 & 1    & 25  & 25 \\ \hline
\end{array} \]
The table shows that neither ordinary homology nor intersection homology are
stable under the smoothing of $V$. The homology of the (middle-perversity) intersection space $IV$
of $V$ has been calculated in \cite{banagl-intersectionspaces} and turns out to be
\[ \rk HI_2 (V)=1,~ \rk HI_3 (V)=204,~ \rk HI_4 (V)=1. \]
This coincides with the above Betti numbers of the smooth deformation $V_s$,
as predicted by our Stability Theorem \ref{thm1} and Remark \ref{gen}. Moreover, formula (\ref{b1}) yields: 
\[ \rk HI_1 (V)=124=\rk HI_5 (V). \]
By using the fact that the small resolution $\ensuremath{\widetilde{V}}$ of $V$ is a Calabi-Yau $3$-fold, we also get that 
\[ \rk IH_1 (V)=0=\rk IH_5 (V). \]
The Euler characteristic identity of Corollary \ref{prop1} is now easily seen to be satisfied.
\end{example}

\begin{example} (cf. \cite{GMS,Ro06}) \\
Let $V \subset \mathbb{P}^4$ be the generic quintic threefold containing the plane $\pi:=\{ z_3=z_4=0 \}$.
The defining equation for $V$ is: \[ z_3 g(z_0,\cdots, z_4) +z_4 h(z_0,\cdots, z_4)=0, \] where $g$ and $h$ are generic homogeneous polynomials of degree $4$. The singular locus of $V$ consists of:
\[ {\rm Sing}(V)=\{ [z] \in  \mathbb{P}^4 ~|~ z_3=z_4=g(z)=h(z)=0 \}=\{ 16 \ {\rm nodes} \}. \] The $16$ nodes of $V$ can be simultaneously resolved by blowing-up $\mathbb{P}^4$ along the plane $\pi$. The proper transform $\ensuremath{\widetilde{V}}$ of $V$ under this blow-up is a small resolution of $V$ (indeed, the fiber of the resolution $\ensuremath{\widetilde{V}} \to V$ over each $p \in {\rm Sing}(V)$ is a $\mathbb{P}^1$), and a smooth Calabi-Yau threefold. In particular, $IH_\ast (V) \cong H_\ast (\ensuremath{\widetilde{V}})$. A smoothing of $V$ is given as in the above example by the generic quintic threefold in $\mathbb{P}^4$, which we denote by $V_s$ ($s \neq 0$). Note that the passage from $V_s$ to $\ensuremath{\widetilde{V}}$ (via $V$) is a non-trivial conifold transition, as $b_2(V_s)=1$ and $b_2(\ensuremath{\widetilde{V}})=2$, i.e., the two Calabi-Yau manifolds $V_s$ and $\ensuremath{\widetilde{V}}$ cannot be smooth fibers of the same analytic family. The information summarized thus far, together with \cite{Ro06,Ro10} yield the following calculation of  ranks ($s \not= 0$):
\[ \begin{array}{|c|c|c|c|} \hline
i & \rk H_i (V_s) & \rk H_i (V) & \rk IH_i (V) \\ \hline \hline
2 & 1     & 1   & 2 \\ \hline
3 & 204 & 189 & 174 \\ \hline
4 & 1    & 2  & 2 \\ \hline
\end{array} \]
Again, neither ordinary homology nor intersection homology are
stable under the smoothing of $V$.
Since $V$ has only nodal singularities, the local monodromy operators are trivial. Therefore, by our Stability Theorem \ref{thm1} and Remark \ref{gen} (see also \cite{banagl-intersectionspaces}[Sect.3.7]), we can compute:
\[ \rk HI_1 (V)=15=\rk HI_5 (V). \]
\[ \rk HI_2 (V)=1=\rk HI_4 (V). \]
\[ \rk HI_3 (V)=204. \]
\[ \rk HI_6 (V)=0. \]
By using the fact that the small resolution $\ensuremath{\widetilde{V}}$ of $V$ is a Calabi-Yau threefold, we also get: 
\[ \rk IH_1 (V)=0=\rk IH_5 (V). \]
Finally, the Euler characteristic identity of Corollary \ref{prop1} reads as: $-232+168=-2 \cdot 32$.
\end{example}


\bibliographystyle{amsalpha}

\providecommand{\bysame}{\leavevmode\hbox to3em{\hrulefill}\thinspace}
\providecommand{\MR}{\relax\ifhmode\unskip\space\fi MR }
\providecommand{\MRhref}[2]{%
  \href{http://www.ams.org/mathscinet-getitem?mr=#1}{#2}
}
\providecommand{\href}[2]{#2}

\end{document}